\documentclass[oneside,english]{amsart}
\usepackage[T1]{fontenc}
\usepackage[latin9]{inputenc}
\usepackage{amsthm}
\usepackage{amsaddr}
\usepackage{amstext}
\usepackage{amssymb}
\usepackage{hyperref}

\usepackage{tikz,epsfig,color,graphicx,epstopdf,subfig}
\usetikzlibrary{positioning,chains,fit,shapes,calc, fit, decorations.pathreplacing}

\usepackage{makecell}
\newcolumntype{x}[1]{>{\centering\arraybackslash}p{#1}}

\makeatletter

\theoremstyle{plain}
\newtheorem{thm}{\protect\theoremname}[section]
  \theoremstyle{plain}
  \newtheorem{cor}[thm]{\protect\corollaryname}
  \theoremstyle{plain}
  \newtheorem{conjecture}[thm]{\protect\conjecturename}
  \theoremstyle{plain}
  \newtheorem{question}[thm]{\protect\questionname}
  \theoremstyle{definition}
  \newtheorem{defn}[thm]{\protect\definitionname}
  \theoremstyle{plain}
  \newtheorem{prop}[thm]{\protect\propositionname}
  \theoremstyle{plain}
  \newtheorem{lem}[thm]{\protect\lemmaname}
  \theoremstyle{plain}

\makeatother

\usepackage{babel}
  \providecommand{\algorithmname}{Algorithm}
  \providecommand{\conjecturename}{Conjecture}
  \providecommand{\corollaryname}{Corollary}
  \providecommand{\definitionname}{Definition}
  \providecommand{\lemmaname}{Lemma}
  \providecommand{\propositionname}{Proposition}
  \providecommand{\questionname}{Question}
\providecommand{\theoremname}{Theorem}

\begin{document}

%\linenumbers % COMMENT OUT for no line number

\def\COMMENT#1{}
%\let\COMMENT=\footnote% COMMENT OUT for clean output

% Probabilistic stuff

\global\long\def\E{\mathbb{E}}

\global\long\def\P{\mathbb{P}}

\global\long\def\var{\textnormal{var}}

\global\long\def\cov{\textnormal{cov}}

% Number systems

\global\long\def\C{\mathbb{C}}

\global\long\def\N{\mathbb{N}}

\global\long\def\Q{\mathbb{Q}}

\global\long\def\R{\mathbb{R}}

\global\long\def\Z{\mathbb{Z}}

% Number theory

\global\long\def\mod{\;(\textnormal{mod}\;}

\global\long\def\Mod{\ \textnormal{mod}\ }

\global\long\def\f{\mathbb{F}}

% Linear algebra stuff

\global\long\def\norm#1{\left|\left|#1\right|\right|}

\global\long\def\ip#1#2{\langle#1,#2\rangle}

\global\long\def\per{\textnormal{per}}

\global\long\def\det{\textnormal{det}}

\global\long\def\tr{\textnormal{Tr}}

% Lattices

\global\long\def\join{\vee}

\global\long\def\meet{\wedge}

% Generally useful stuff

\global\long\def\supp{\textnormal{supp }}

\global\long\def\re{\textnormal{Re }}

\global\long\def\im{\textnormal{Im }}

\global\long\def\floor#1{\left\lfloor #1\right\rfloor }

\global\long\def\ceil#1{\left\lceil #1\right\rceil }

\global\long\def\topcirc#1{\overset{\circ}{#1}}

\global\long\def\diam{\textnormal{diam}}

\global\long\def\rad{\textnormal{rad}}

\renewcommand{\labelenumi}{(\roman{enumi})}

\global\long\def\F{\mathcal{F}}

\title{On-line Ramsey numbers of paths and cycles}

\author{Joanna Cyman}
\address{Department of Technical Physics and Applied Mathematics, Gda\'nsk University of Technology, Narutowicza 11/12, 80-952 Gda\'nsk, Poland}
\email{joana@mif.pg.gda.pl}

\author{Tomasz Dzido}
\address{Institute of Informatics, University of Gda\'nsk, Wita Stwosza 57, 80-952 Gda\'nsk, Poland}
\email{tdz@inf.ug.edu.pl}

\author{John Lapinskas}
\email{lapinskas@cs.ox.ac.uk}
\address{Department of Computer Science, University of Oxford, Wolfson Building, Parks Road, Oxford, OX1~3QD, United Kingdom}
%\thanks{The research leading to these results has received funding from the European Research Council under the European Union’s Seventh Framework Programme (FP7/2007-2013) ERC grant agreement no. 334828. The paper reflects only the authors’ views and not the views of the ERC or the European Commission. The European Union is not liable for any use that may be made of the information contained therein.}

\date{\today}
\author{Allan Lo}
\email{s.a.lo@bham.ac.uk}
\address{School of Mathematics, University of Birmingham, Edgbaston, Birmingham, B15~2TT, United Kingdom}
\thanks {The research leading to these results has received funding from the European Research Council under the European Union's Seventh Framework Programme (FP7/2007-2013) ERC grant agreements no. 258345 (A. Lo) and 334828 (J. Lapinskas). The paper reflects only the authors'  views and not the views of the ERC or the European Commission. The European Union is not liable for any use that may be made of the information contained therein.}

\maketitle
\begin{abstract}
Consider a game played on the edge set of the infinite clique by two
players, Builder and Painter. In each round, Builder chooses an edge
and Painter colours it red or blue. Builder wins by creating either
a red copy of $G$ or a blue copy of $H$ for some fixed graphs $G$
and $H$. The minimum number of rounds within which Builder can win, assuming both
players play perfectly, is the \emph{on-line Ramsey number} $\tilde{r}(G,H)$.
In this paper, we consider the case where $G$ is a path $P_k$.
We prove that $\tilde{r}(P_3,P_{\ell+1}) = \lceil 5\ell/4\rceil = \tilde{r}(P_3,C_{\ell})$ for all $\ell \ge 5$,
and determine $\tilde{r}(P_{4},P_{\ell+1})$ up to an additive constant for all $\ell \ge 3$.
We also prove some general lower bounds for on-line Ramsey numbers of the form $\tilde{r}(P_{k+1},H)$.
\end{abstract}

\section{Introduction\label{sec:Introduction}}

\COMMENT{TODO: Sort out introduction. Strip out references to Table 1 in body of proof. Final pass of error checking. -JAL 25/10}\COMMENT{LaTeX does weird things with spacing, so that e.g. the code "\{\} \{\} \{\} " will give rise to three spaces in the output rather than just one. This means that adding spaces after {\textbackslash}COMMENTs that aren't at the end of a paragraph can do bad things to the formatting when they're removed from the document. I've taken out all the examples I found. -JAL 16/08}Ramsey's theorem \cite{ramsey} states that for all $k\in\N$, there
exists $t \in \N$ such that any red-blue edge colouring of a clique $K_{t}$
contains a monochromatic clique of order~$k$. We call the least such
$t$ the \emph{$k^{\text{th}}$ Ramsey number}, and denote it by $r(k)$. Ramsey
numbers and their generalisations have been a fundamentally important
area of study in combinatorics for many years. Particularly well-studied
are Ramsey numbers for graphs. Here the \emph{Ramsey number} of two graphs $G$ and $H$, denoted by $r(G,H)$,
is the least $t$ such that any red-blue edge colouring of $K_{t}$
contains a red copy of $G$ or a blue copy of $H$. See e.g. \cite{radziszowski}
for a survey of known Ramsey numbers.\COMMENT{I think we should have some indication of what the citation is about... -JAL 25/10}

An important generalisation of Ramsey numbers, first defined by Erd\H{o}s, Faudree, Rousseau and Schelp \cite{size}, is as follows. Let $G$ and $H$ be two graphs. We say that a graph~$K$ has the \emph{$(G,H)$-Ramsey property} if any red-blue edge colouring of $K$ must contain either a red copy of $G$ or a blue copy of $H$. Then the \emph{size Ramsey number} $\hat{r}(G,H)$ is given by the minimum number of edges of any graph with the $(G,H)$-Ramsey property.

In this paper, we consider the following related generalisation 
defined independently by Beck~\cite{beck} and Kurek and
Ruci\'{n}ski~\cite{rucinski}. Let $G$ and $H$ be two graphs. Consider
a game played on the edge set of the infinite clique $K_{\N}$ with
two players, Builder and Painter. In each round of the game, Builder
chooses an edge and Painter colours it red or blue. Builder wins by
creating either a red copy of $G$ or a blue copy of $H$, and wishes
to do so in as few rounds as possible. Painter wishes to delay Builder
for as many rounds as possible. (Note that Painter may not delay Builder indefinitely -- for example,
Builder may simply choose every edge of $K_{r(G,H)}$.)
The \emph{on-line Ramsey number} $\tilde{r}(G,H)$ is the minimum number of rounds it takes
Builder to win, assuming that both Builder and Painter play optimally.
We call this game the \emph{$\tilde{r}(G,H)$-game}, and write $\tilde{r}(G)=\tilde{r}(G,G)$.
Note that $\tilde{r}(G,H)\ge e(G)+e(H)-1$ for all graphs $G$ and $H$, as
Painter may simply colour the first $e(G)-1$ edges red and all subsequent
edges blue. It is also clear that $\tilde{r}(G,H) \le \hat{r}(G,H)$.%
\COMMENT{This was previously $\tilde{r}(G,H) \ge \hat{r}(G,H)$, which is the wrong way round. -JAL 25/10}%
\COMMENT{Removed the paragraph about applications in computer science after realising that the Erd\H{o}s-Szekeres bound
already gives a perfectly good algorithm for finding a copy of $G$ or $H$ while uncovering few edges. Online Ramsey
numbers are still technically best possible in terms of edges uncovered, but the difference isn't particularly interesting from a practical perspective.
Derp. -JAL 16/08}

On-line Ramsey theory has been well-studied. The best known bounds for $\tilde{r}(K_t)$ are given by
\[\frac{r(t)-1}{2} \le \tilde{r}(K_t) \le t^{-c\frac{\log t}{\log\log t}}4^t,\]
where $c$ is a positive constant. The lower bound is due to Alon (and was first published in a paper of Beck~\cite{beck}), and the upper bound is due to Conlon~\cite{conlon}.
Note that these bounds are similar to the best known bounds for classical Ramsey numbers $r(t)$, although Conlon also proves in~\cite{conlon} that
\[\tilde{r}(K_t) \le C^{-t}\binom{r(t)}{2}\]
for some constant $C>1$ and infinitely many values of~$t$, which gives positive evidences supporting a conjecture of Kurek and Ruci\'{n}ski~\cite{rucinski} that $\tilde{r}(K_t) = o (r(t)^2)$.
For general graphs $G$,
the best known lower bound for $\tilde{r}(G)$ is given by Grytczuk, Kierstead and Pra\l{}at~\cite{grytczuk}.

\begin{thm}
For graphs $G$, we have $\tilde{r}(G) \ge \beta(G) (\Delta(G) -1)/2 + e(G)$, where $\beta(G)$ denotes the vertex cover number of $G$.
\end{thm}

Various general strategies for Builder and Painter have also been studied. For example, consider the following strategy for Builder in the $\tilde{r}(G,H)$-game. Builder chooses a large but finite set of vertices in $K_\N$, say a set of size $n\in\N$, with $n\ge r(G,H)$. Then Builder chooses the edges of the induced $K_n$ in a uniformly random order, allowing Painter to colour each edge as they wish, until the game ends. This strategy was analysed for the $\tilde{r}(K_3)$-game by Friedgut, Kohayakawa, R\"{o}dl, Ruci\'{n}ski and Tetali~\cite{friedgut}, and for the more general $\tilde{r}(G)$-game by Marciniszyn, Sp\"{o}hel and Steger~\cite{marc1,marc2}.

Finally, let $\tilde{r}_{\chi}(G)$-game be the $\tilde{r}(G)$-game in which Builder is forbidden to uncovering a graph with chromatics number greater than $\chi(G)$. 
Grytczuk, Ha{\l}uszczak and Kierstead~\cite{grytczuk2} proved that Builder can win the $\tilde{r}_{\chi}(G)$-game.
Kierstead and Konjevod~\cite{kierstead} proved the hypergraph generalisation.

%Finally, it is interesting to consider the results of possible restrictions of Builder's strategy. Recall that Builder can win a conventional $\tilde{r}(G)$-game by choosing a clique of size $r(G)$. Setting aside our assumption that Builder plays optimally, it is not immediate that Builder can win at all without doing so. However, Grytczuk, Ha{\l}uszczak and Kierstead~\cite{grytczuk2} proved (among other things) that if $\chi(G)\le k$, then Builder can win the $\tilde{r}(G)$-game without uncovering a graph with chromatic number greater than $k$. 
%Kierstead and Konjevod~\cite{kierstead} consider similar questions for a generalisation of the $\tilde{r}(G,H)$-game to hypergraphs. 

Given the known bounds on $\tilde{r}(K_t)$, it is not surprising that determining on-line Ramsey numbers exactly has proved even more difficult than determining classical Ramsey numbers exactly, and very few results are known.\COMMENT{I think we should say "more difficult" rather than "as difficult" here, since IIRC it's fairly easy to determine $r(P_{k+1},P_{\ell+1}$... -JAL 25/10}
A significant amount of effort has been focused on the special 
case where $G$ and $H$ are paths. Grytczuk, Kierstead and Pra\l{}at~\cite{grytczuk} and Pra\l{}at~\cite{pralat,pralat2}
have determined $\tilde{r}(P_{k+1},P_{\ell+1})$ exactly when $\max\{k,\ell\}\le8$ (where $P_s$ is a path on $s$ verices).
In addition, Beck~\cite{beck2} has proved that 
the size Ramsey number $\hat{r}(P_k,P_k)$ is linear in $k$. (The best known upper bound, due to Dudek and Pra{\l}at~\cite{newsize}, is $\hat{r}(P_k,P_k) \le 137k$.)
The best known bounds on $\tilde{r}(P_{k+1},P_{\ell+1})$ were proved in~\cite{grytczuk}.

\begin{thm}
\label{thm:oldbound} 
For all $k, \ell \in\N$, we have $k + \ell -1\le\tilde{r}(P_{k+1}, P_{\ell+1} )\le 2 k + 2\ell -3$.
\end{thm}

In general, it seems difficult to bound on-line Ramsey numbers $\tilde{r}(G,H)$ below.\COMMENT{"Bound... below" is correct here -- no "from" is needed. -JAL 25/10}
One of the major difficulties in doing so is the variety of possible strategies for
Builder. We present a strategy for Painter which mitigates this problem somewhat.

\begin{defn}
Let $\F$ be a family of graphs. We define the \emph{$\F$-blocking
strategy} for Painter as follows. Write $R_{i}$ for the graph consisting of all uncovered red edges
immediately before the $i$th move of the game, and write $e_{i}$
for the $i$th edge chosen by Builder. Then Painter colours $e_{i}$
red if $R_{i}+e_{i}$ is $\F$-free, and blue otherwise.
(Recall that a graph is \emph{$\F$-free} if it contains no graph in $\F$ as a subgraph.)
\end{defn}

In an $\tilde{r}(G,H)$-game, it is natural to consider $\F$-blocking strategies with ${G\in\F}$.
For example, if $\F=\{G\}$, then the $\F$-blocking strategy for
Painter consists of colouring every edge red unless doing so would
cause Painter to lose the game.
If Painter is using an $\F$-blocking strategy, one clear strategy
for Builder would be to construct a red $\mathcal{F}$-free graph, then
use it to force a blue copy of $H$ in $e(H)$ moves.
We will show that this is effectively Builder's only strategy (see Proposition~\ref{prop:blocking}), and thus to bound $\tilde{r}(G,H)$ below
it suffices to prove that no small red $\F$-free graph can be used to force a blue copy of $H$.
We use this technique to derive some lower bounds for on-line Ramsey
numbers of the form $\tilde{r}(P_{k+1},H)$, taking $\F=\{P_{k+1}\}\cup\{C_{i}:i\ge3\}$.\COMMENT{Looking at this, we don't actually know the bound is strongest when $H$ is close to being a regular bipartite graph unless Builder is forced to use our specific $\F$-blocking strategy, so I've taken that out. -JAL 16/08} 
\begin{thm}
\label{thm:lowerbound}Let $k,\ell \in\N$ with $k\ge2$. Let $H$ be a graph on $|H|$ vertices with $\ell$ edges and let $\Delta=\Delta(H)$. 
Then
\begin{align*}
\tilde{r}(P_{k+1},H)  \ge \begin{cases}
(2\Delta+1) \ell / (2\Delta) & \textnormal{if } k=2,\\
(5\Delta +4) \ell /( 5\Delta) & \textnormal{if } k=3,\\
(\Delta+1) \ell / \Delta  & \textnormal{if } k\ge4.
\end{cases}
\end{align*}
Moreover, if $H$ is connected and $k \ge 4 $, then 
\begin{align*}
\tilde{r}(P_{k+1},H) \ge 
{(\Delta+1)\ell}/{\Delta} + \min \left\{ k/2-2, |H| -1 \right\} .
\end{align*}
\end{thm}

For $k=2$, we show that if $H = P_{\ell +1}$ for $\ell \ge 2$ or $H = C_{\ell}$ for $\ell \ge 5$, then the bound on $\tilde{r}(P_3,H)$ given by Theorem~\ref{thm:lowerbound} is tight. 
\begin{thm}
\label{thm:pathresultsk=2}
For all $\ell\ge 2 $, we have $\tilde{r}(P_{3},P_{\ell+1})  = \ceil{{5\ell}/{4}}$.
Also,
\[
\tilde{r}(P_{3},C_{\ell})=\begin{cases}
\ell+2 & \textnormal{if } \ell = 3,4, \\
\ceil{{5\ell}/{4}} & \textnormal{if } \ell\ge5.
\end{cases}
\]
\end{thm}
%This theorem is the first non-trivial exact result on $\tilde{r}(G,H)$ for a fixed $G$ and a family of graphs $H$.
Furthermore, for $k =3$, we determine $\tilde{r}(P_4,P_{\ell +1})$
up to an additive constant for all $\ell \ge 3$. 
\begin{thm}
\label{thm:pathresultsk=3}
For all $\ell \ge 3$, we have $(7\ell+2)/5 \le \tilde{r}(P_{4},P_{\ell+1})  \le (7\ell+52)/5$.
\end{thm}
Our proof of the upper bound for $k = 3$ is complicated, so the proof is included in the Appendix. 
The lower bound follows from Lemma~\ref{lem:p4forestscaffold}, a simple extension of the proof of Theorem~\ref{thm:lowerbound}, and we believe that it is tight.

\begin{conjecture}
\label{con:p4con}
For all $\ell \ge 3$, we have $\tilde{r}(P_4, P_{\ell+1}) = \lceil(7\ell+2)/5\rceil$.
\end{conjecture}

By Theorems~\ref{thm:pathresultsk=2} and~\ref{thm:pathresultsk=3}, we have 
\begin{align*}
\lim_{\ell \rightarrow \infty} \tilde{r}(P_{3},P_{\ell+1}) / \ell &= 5/4, \\
\lim_{\ell \rightarrow \infty} \tilde{r}(P_{4},P_{\ell+1}) / \ell &= 7/5.
\end{align*}
On the other hand, for all fixed $k\ge 4$, Theorems~\ref{thm:oldbound} and~\ref{thm:lowerbound} imply that 
\[3/2 \le \liminf_{\ell \rightarrow \infty} \tilde{r}(P_{k+1},P_{\ell+1})/\ell \le \limsup_{\ell \rightarrow \infty} \tilde{r}(P_{k+1},P_{\ell+1}) / \ell \le 2,\] 
and we make the following conjecture. 
 
\begin{conjecture}
\label{con:maincon}
For $k \ge 4$, $\lim_{\ell \rightarrow \infty} \tilde{r}(P_{k+1},P_{\ell+1}) / \ell = 3/2$.
Moreover, for all $\ell\ge k \ge 4$, we have $\tilde{r}(P_{k+1},P_{\ell+1})= \ceil{ 3\ell /{2} }+k-3$.
In particular, we have $\tilde{r}(P_{k+1})=\ceil{5k/2}-3$ for $k \ge 4$.
\end{conjecture}

Note that Conjecture~\ref{con:maincon} would imply Conjecture~4.1 of \cite{pralat2}. 
Conjectures~\ref{con:p4con} and~\ref{con:maincon} have been confirmed for $ \ell \le 8$ by Pra{\l}at~\cite{pralat}, using a high-performance computer cluster.

Finally, we give some bounds on $\tilde{r}(C_{4},P_{\ell+1})$.
\begin{thm} \label{prop:C4}
For $\ell \ge 3$, we have $2 \ell \le \tilde{r}( C_{4},P_{\ell +1}) \le 4\ell-4$.
Moreover,  $\tilde{r}( C_{4},P_{4}) = 8$.
\end{thm}

Many of the lower bounds above follow from Theorem~\ref{thm:lowerbound}, and all of them follow  from analysing $\F$-blocking strategies. In particular, we obtain tight lower bounds on $\tilde{r}(P_3, P_{\ell+1})$ and $\tilde{r}(P_3, C_\ell)$ in this way, as well as a lower bound on $\tilde{r}(P_4, P_{\ell+1})$ which matches Conjecture~\ref{con:p4con}. We are therefore motivated to ask the following question.

\begin{question}
For which graphs $G$ and $H$ does there exist a family $\mathcal{F}$ of graphs such that the $\mathcal{F}$-blocking strategy is optimal for Painter in the $\tilde{r}(G,H)$-game?
\end{question}

The paper is laid out as follows.
In Section~\ref{sec:Lower-bounds}, we prove Theorem~\ref{thm:lowerbound}.
We prove Theorem~\ref{thm:pathresultsk=2} in Sections~\ref{sec:exactpathcalcs} and~\ref{sec:exactcyclecalcs} (see Theorem~\ref{thm:P3-Pl-exact}, Proposition~\ref{prop:smallcycleexact} and Theorem~\ref{thm:P3-Cl-exact}).
Finally, in Section~\ref{sec:C4} we prove Theorem~\ref{prop:C4}.
The proof of Theorem~\ref{thm:pathresultsk=3} is in the Appendix.

%%%%%%%%%%%%%%%%%%%%%%%%%%%%%%%%%%%%%%%%%%%%%%%%%%%%%%%%%%%%%%%%%%%%%%%%%%%%%%%%%%%%%%%%%%%%%

\section{Notation and conventions}\label{sec:notation}

We write $\N$ for the set $\{1,2,\dots\}$ of natural numbers, and $\N_0:=\N\cup\{0\}$.

Suppose $P=v_{1}\dots v_{k}$ and $Q=w_{1}\dots w_{\ell}$ are paths.
If $i<j$, we write $v_{i}Pv_{j}$ (or $v_{j}Pv_{i}$) for the subpath
$v_{i}v_{i+1}\dots v_{j}$ of $P$. We also write $PQ$ for the concatenation
of $P$ and $Q$. 
% If this notation would result in duplicate vertices,
% we remove them. We use a similar convention for appending edges to
% a path. 
For example, if $i<j$ and $i'<j'$ then $uv_{i}Pv_{j}yw_{i'}Qw_{j'}$ denotes the
path $uv_{i}v_{i+1}\dots v_{j}yw_{i'}w_{i'+1}\dots w_{j'}$.

If $G$ is a graph, we will write $|G|$ for the number of vertices of $G$ and $e(G)$ for the number of edges of $G$.

In the context of an $\tilde{r}(G,H)$-game, an \emph{uncovered edge}
is an edge of $K_{\N}$ that has previously been chosen by Builder,
and a \emph{new vertex} is a vertex in $K_{\N}$ not incident to any
uncovered edge.

Many of our lemmas say that in an $\tilde{r}(G,H)$-game, given a finite coloured graph $X \subseteq K_{\N}$, 
Builder can force Painter to construct a coloured graph $Y\subseteq K_\N$ satisfying some desired property. We will often
apply such a lemma to a finite coloured graph $X' \supsetneq X$, and in these cases we will implicitly require $V(Y) \cap V(X') \subseteq V(X)$.
(Intuitively, when Builder chooses a new vertex while constructing $Y$, it should be new with respect to $X'$ rather than $X$.)
This is formally valid, since we may apply the lemma to an $\tilde{r}(G,H)$-game on the board $K_{\N} - (V(X') \setminus V(X))$ and have Builder
choose the corresponding edges in $K_{\N}$.

For technical convenience, we allow Builder to ``waste'' a round
in the $\tilde{r}(G,H)$-game by choosing an uncovered edge. If he does so, the round contributes to the duration of the game
but the edge Builder chooses is not recoloured. Since such a move is never optimal for Builder, the definition of $\tilde{r}(G,H)$ is not affected.

\section{\label{sec:Lower-bounds}General lower bounds}

Our aim is to bound $\tilde{r}(G,H)$ below for graphs $G$ and $H$.
In this section, Painter will always use an $\mathcal{F}$-blocking strategy for some family $\mathcal{F}$ of graphs with $G \in \F$.
Hence, as we shall demonstrate in Proposition~\ref{prop:blocking} below, Builder's strategy boils down to choosing a red graph with which
to force a blue copy of $H$.

\begin{defn}
Let $\F$ be a family of graphs and let $R\subseteq K_{\N}$ be an $\F$-free graph.
We say that an edge $e\in K_{\N}-R$ is \emph{$(R,\F)$-forceable} if $R+e$ is not $\F$-free.
We say a graph $H$ is \emph{$(R,\F)$-forceable} if there exists
$H'\subseteq K_{\N}-R$ with $H'$ isomorphic to $H$ such that every edge $e\in E(H')$
is $(R,\F)$-forceable. We call $H'$ an \emph{$(R,\F)$-forced copy of $H$}.
If $R$ and $\F$ are clear from context, we will omit `$(R,\F)$-'.
\end{defn}
% l
\begin{defn}
Let $\F$ be a family of graphs and let $H$ be a graph. We say a graph $R\subseteq K_{\N}$
is an \emph{$\F$-scaffolding for $H$} if the following
properties hold.
\begin{enumerate}
\item $R$ is $\F$-free.
\item $H$ is $(R,\F)$-forceable.
\item $R$ contains no isolated vertices.
\end{enumerate}
\end{defn}

\begin{prop}
\label{prop:blocking}
Let $G$ and $H$ be graphs.
Let $\F$ be a family of graphs with $G \in \F$.
Suppose every $\F$-scaffolding for $H$ has at least $m$ edges. 
Then $\tilde{r} (G,H) \ge m + e(H)$.
\end{prop}

\begin{proof}
Consider an $\tilde{r}(G,H)$-game in which Painter uses an $\mathcal{F}$-blocking strategy.
Further suppose Builder wins by claiming edges $e_{1},\dots,e_{r}$.
Since Builder choosing an edge which Painter colours blue has no effect on
Painter's subsequent choices, without loss of generality we may assume
that there exists $i$ such that Painter colours $e_{1},\dots,e_{i}$
red and $e_{i+1},\dots,e_{r}$ blue.
Let $R\subseteq K_{\N}$ be the subgraph with edge set $\{e_{1},\dots,e_{i}\}$, and let $B\subseteq K_{\N}$
be the subgraph with edge set $\{e_{i+1},\dots,e_{r}\}$.
Thus $R$ is the uncovered red graph and $B$ is the uncovered blue graph.

We will show that $R$ is an $\F$-scaffolding for $H$.
First note that $R$ is $\F$-free by Painter's strategy, and $R$ has no isolated vertices by definition.
Moreover, since $G \in \F$ and Builder wins, there exists $H'\subseteq B$ with $H'$ isomorphic to $H$.
So $e(B) \ge e(H)$.
Moreover, by Painter's strategy all edges in $B$ must be $(R, \F)$-forceable, so $H$ is $(R, \F)$-forceable.
Hence $R$ is an $\F$-scaffolding for $H$, so $e(R) \ge m$.
Therefore, Builder wins in $r \ge e(R) + e(B) \ge m + e(H)$ rounds.
\end{proof}

Therefore, to bound $\tilde{r}(G,H)$ below, it suffices to bound the number of edges in an $\F$-scaffolding for $H$ below for some family $\F$ of graphs with $G\in\F$.
We first use Proposition~\ref{prop:blocking} to bound $\tilde{r}( C_{k}, H )$ for connected graphs~$H$.

\begin{lem} \label{lem:Ckscaf}
Let $H$ be a connected graph.
Then every $\{ C_i: i \ge 3 \}$-scaffolding for~$H$ has at least $|H|-1$ edges. 
Moreover, $\tilde{r}( C_{k}, H ) \ge |H| + e(H) -1$ for all $k \ge 3$.
\end{lem}

\begin{proof}
Let $R$ be a $\{ C_k \}$-scaffolding for~$H$ with $e(R)$ minimal.
Note that each $(R,\{C_{k}\})$-forceable edge must lie entirely in a component of~$R$.
Since $H$ is connected, $R$ is connected and $|R| \ge |H|$.
Hence, $e(R) \ge  |H| - 1$. 

By Proposition~\ref{prop:blocking}, $\tilde{r}( C_{k}, H ) \ge |H| + e(H) -1$.
\end{proof}

To prove Theorem~\ref{thm:lowerbound}, we set $G=P_{k+1}$ and $\F=\{P_{k+1}\}\cup\{C_{i}:i\ge3\}$.
Thus an $\F$-free graph is a forest whose components have diameter less than~$k$.\COMMENT{Was "a forest with diameter less than $k$", but this would have to be a spanning tree. -JAL 16/08} Lemma~\ref{lem:forestscaffold} gives a lower bound on the number of edges in an $\F$-scaffolding for $H$. 
%Theorem~\ref{thm:lowerbound} then follows immediately from Lemma~\ref{lem:forestscaffold} and Proposition~\ref{prop:blocking}.

Note that replacing $\F$ by $\{P_{k+1}\}$ and attempting a similar proof yields a worse lower bound in some cases. For example, taking $H=P_{2k+1}$ with $k\ge 3$, if Painter follows the $\{P_{k+1}\}$-blocking strategy then Builder can win in $3k$ moves by first constructing a red~$C_k$.\COMMENT{Note that I don't think the choice of $\F$ matters if e.g. $H$ is a clique. -JAL 30/09}

We will see in the proof of Lemma~\ref{lem:forestscaffold} that if $R$ is a red $\F$-free graph with no isolated vertices, and $X\subseteq V(R)$ is the set of endpoints of $P_k$'s in $R$, then Builder may force at most $\Delta(H)(|R|+|X|)$ edges of $H$ using $R$. It will therefore be very useful to bound $|R|+|X|$ above in terms of $e(R)$, first in the special case where $R$ is a tree (see Lemma~\ref{lem:treescaffold}) and then in general (see Lemma~\ref{lem:forestscaffold1}).\COMMENT{This reference was previously incorrect. -JAL 15/10}

\begin{lem}
\label{lem:treescaffold}Let $k,m\in\N$ with $k\ge 2$. Let $R$ be
a $P_{k+1}$-free tree with $m$ edges. Let $X$ be the
set of endpoints of $P_{k}$'s in $R$.
If $X \ne \emptyset$, then $|R|+|X|\le2m-k+4$.
%Moreover, $|R|+|X|\le\max\{2m-k+4,2m\}$.
%Moreover, if $X\ne\emptyset$, then $|R|+|X|\le2m-k+4$.
\end{lem}

\begin{proof}
We claim that if $x\in X$, then $x$ is a leaf of $R$. Indeed, let
$P$ be a $P_{k}$ with one endpoint equal to $x$. Let $y\in V(P)$
be the neighbour of $x$ in $P$, and suppose $xz\in E(R)$ for some
$z\ne y$. Then either $z\in V(P)$ and $xzPx$ is a cycle in $R$,
or $z\notin V(P)$ and $Pxz$ is a $P_{k+1}$ in $R$ -- both are
contradictions. Hence if $x\in X$, then $x$ is a leaf. But since $X\ne\emptyset$,
$R$ contains a $P_{k}$ and hence at least $k-2$ vertices of degree
greater than 1.\COMMENT{This grammar is correct. -JAL 06/11} Hence
\[
|R|+|X|\le|R|+|R|-(k-2)=2m-k+4,
\]
and the proposition follows.
%If $X=\emptyset$, then $|R|+|X|=|R|= m+1\le2m$ and we are done.
%We may therefore assume that $X\ne\emptyset$. 
\end{proof}

\begin{lem}
\label{lem:forestscaffold1}
Let $k,m\in\N$ with $k\ge 2$. Let $R$ be a $P_{k+1}$-free forest with $m$ edges and no isolated vertices. Let $X$ be the set of all
endpoints of $P_{k}$'s in $R$.
Then
\[
|R|+|X|\le\begin{cases}
4m & \textnormal{if } k=2,\\
5m/2 & \textnormal{if } k=3,\\
2m & \textnormal{if } k\ge4.
\end{cases}
\]
Moreover, if $k \ge 4$ and there exists an edge $e$ such that $R + e$ contains a $P_{k+1}$, then $|R|+|X|\le
2m - k+4$.
\end{lem}

\begin{proof}
Let $R_{1},\dots,R_{r}$ be the components of $R$. Let $m_{i}=e(R_{i})$
and $X_{i}=X\cap V(R_{i})$ for all $1 \le i \le r$.
If $k=2$, then $R$ is a disjoint union of $m$ edges and the result is immediate. 
%\begin{align*}
%|R|+|X|=\sum_{i=1}^{r}(|R_{i}|+|X_{i}|)\le \sum_{i=1}^{r}( 2 m_{i}+2 ) = 2( m + r)  \le 4m.
%\end{align*}

Suppose $k=3$. Without loss of generality, let $R_{1},\dots,R_{r'}$
be those components of $R$ which consist of a single edge. (Note
that we may have $r'=0$.)
Then $m = r' + \sum_{i =r'+1 }^{r} m_i$ and $r - r' \le m/2$.
Then by Lemma~\ref{lem:treescaffold}
we have
\begin{align*}
|R|+|X| & =\sum_{i=1}^{r'}|R_{i}|+\sum_{i=r'+1}^{r}(|R_{i}|+|X_{i}|)
  \le2r'+ \sum_{i=r'+1}^{r} ( 2 m_{i} + 1) \\
 & =2m+r-r'\le 5m/2
\end{align*}
and so the result follows.

Finally, suppose $k \ge 4$.
Let $q$ be the number of components of $R$ containing a~$P_{k}$.
Without loss of generality suppose that $R_1, \dots, R_q$ are the components of $R$ which contain a $P_k$.
For $ q< i \le r$, we have $|R_{i}|+|X_{i}| = |R_i| = m_i+1 \le 2 m_i$. 
 Then by Lemma~\ref{lem:treescaffold} we
have 
\begin{equation}
|R|+|X|=\sum_{i=1}^{r}(|R_{i}|+|X_{i}|)\le \sum_{i=1}^{q}( 2 m_{i}-k+4 ) +\sum_{i=q+1}^r (2m_i) = 2m-q(k-4).\label{eq:forestscaffold1}
\end{equation}
Suppose that there exists an edge $e$ such that $R + e$ contains a~$P_{k+1}$.
If $X \ne \emptyset$, then $q \ge 1$ and so $|R|+|X| \le 2m - k+4$ by (\ref{eq:forestscaffold1}).
Hence we may assume that $X = \emptyset$, and so $e$ is an edge between two vertices of $R$.
It follows that $R$ contains two vertex-disjoint paths of combined length at least $k-1$, and hence that
\[|R|+|X| = |R| = m + r \le m + (m - k + 3) < 2m - k + 4,\]
as desired. The first inequality follows since all edges in a given path must lie in the same component of $R$.
\end{proof}

\begin{lem} \label{lem:forestscaffold}
Let $k,\ell\in\N$ with $k\ge2$.
Let $H$ be a graph with $\ell$ edges and let $\Delta=\Delta(H)$.
Let $\F=\{P_{k+1}\}\cup\{C_{i}:i\ge3\}$.
Suppose $R$ is an $\F$-scaffolding for $H$.
Then, we have
\begin{align*}
e(R) \ge \begin{cases}
\ell / (2\Delta) & \textnormal{if } k=2,\\
4 \ell /( 5\Delta) & \textnormal{if } k=3,\\
\ell / \Delta  & \textnormal{if } k\ge4.
\end{cases}
\end{align*}
Moreover, if $H$ is connected and $k \ge 4$ then $e(R) \ge \min \left\{ \frac{\ell }{\Delta} + \frac{k}2-2, |H|-1 \right\}$.
%\begin{align*}
%e(R) \ge \begin{cases}
%\frac25\left\lceil \frac{2 \ell +1}{d}  \right\rceil& \textnormal{if } k=3,\\
%\min \left\{ \frac{\ell }{d} + \frac{k}2-2, |H|-1 \right\}& \textnormal{if } k\ge4.
%\end{cases}
%\end{align*}
\end{lem}

\begin{proof}
Let $m = e(R)$.
Note that $R$ is a $P_{k+1}$-free forest with $m$ edges and no isolated vertices.
Let $X$ be the set of endpoints of $P_{k}$'s in $R$ and let $Y=V(R)\setminus X$.

We first claim that any $(R,\F)$-forceable edge is either incident to $X$ or internal to $Y$. Suppose not.
Then there exist $y\in Y$ and $z\notin V(R)$ such that $yz$ is a forceable edge.
Let $F\in\F$ be such that $F\subseteq R+e$. Note
that $e\in E(F)$, since $R$ is $\F$-free. Since $d_{R+e}(z)=1$, we
have $F=P_{k+1}$. But then $y$ is an endpoint of a $P_{k}$ in $R$,
contradicting $y\in Y$.

Let $H'$ be a forced copy of $H$. Then $H'$ contains at most $\Delta|X|$
edges incident to $X$, and at most $\Delta|Y|/2$ edges internal to $Y$.
All edges of $H'$ are forceable, so it follows that 
\begin{align}
\ell = e(H') \le \Delta|X|+\frac{\Delta|Y|}{2}= \frac{\Delta(|R|+|X|)}{2}. \label{RX3}
\end{align}
Lemma~\ref{lem:forestscaffold1} and (\ref{RX3}) imply the lemma holds unless $k\ge 4$ and $H$ is connected.

Now suppose $H$ is connected and $k \ge 4$. 
If there exists an edge $e$ such that $R + e$ contains a $P_{k+1}$, then $|R|+|X|\le
2m - k+4$ by Lemma~\ref{lem:forestscaffold1}.
Hence, \eqref{RX3} implies that $m \ge \frac{\ell }{\Delta} + \frac{k}2-2$.
Therefore, we may assume that no such edge exists, and in particular that $X = \emptyset$.
This implies that $R$ is a $\{C_{i}:i\ge3\}$-scaffolding for $H$.
Lemma~\ref{lem:Ckscaf} implies that $m \ge |H| -1$ as required.

\end{proof}

Theorem~\ref{thm:lowerbound} follows immediately from Proposition~\ref{prop:blocking} and Lemma~\ref{lem:forestscaffold}.

We now bound $\tilde{r}(P_{4},P_{\ell+1})$ from below.

\begin{lem} \label{lem:p4forestscaffold}
Let $\ell \in \N$ with $\ell \ge 3$. Then we have $\tilde{r}(P_4, P_{\ell+1}) \ge (7\ell+2)/5$.
\end{lem}
%
%Let $\ell\in\N$ with $\ell\ge3$.
%Let $\F=\{P_{4}\}\cup\{C_{i}:i\ge3\}$.
%Suppose $R$ is an $\F$-scaffolding for $P_{\ell+1}$.
%Then, we have $e(R)\ge 2(\ell +1) /5$.
%Moreover, $\tilde{r}(P_{4},P_{\ell+1}) \ge (7 \ell +2)/5$.

\begin{proof}
Let $\F=\{P_{4}\}\cup\{C_{i}:i\ge3\}$. Let $R$ be an $\F$-scaffolding for $P_{\ell+1}$. 
Let $X$ be the set of endpoints of $P_{3}$'s in $R$, and let $Y=V(R)\setminus X$.
By Lemma~\ref{lem:forestscaffold1} and Proposition~\ref{prop:blocking}, to prove the lemma it suffices to show that $|R| + |X| \ge \ell +1$.

Let $H$ be a forced copy of $P_{\ell+1}$.
Note that any $(R,\F)$-forceable edge is either incident to $X$ or internal to $Y$. 
Note also that $Y \ne \emptyset$. Indeed, if $X = \emptyset$ then this is immediate. If $X \ne \emptyset$, then $R$ is a $P_4$-free forest containing a $P_3$. The central vertex of this $P_3$ cannot be an element of $X$, and is therefore an element of $Y$.

Since $\Delta(H) = 2$, $H$ contains at most $2|X|$ edges incident to $X$. Moreover, since $H$ is a path, $H[Y]$ is a forest and so $e_H(Y) \le |Y| - 1$. It follows that
\[\ell \le 2|X| + |Y| - 1 = |R| + |X| - 1,\]
and hence $|R| + |X| \ge \ell + 1$ as desired.

%If all edges of $H$ are internal to $Y$, then $\ell + 1 = |H| \le |Y| \le |R|+|X|$ as desired.
%We may therefore assume that $X \ne \emptyset$.
%Moreover, this implies that $Y \ne \emptyset$. (Indeed, since $R$ is a $P_4$-free forest only leaves of $R$ can be elements of $X$.
%Since $X \ne \emptyset$, $R$ contains a $P_3$ and hence a non-leaf.)
%Since $H$ is a path, there are at most $|Y|-1$ edges of $H$ internal to $Y$ and so $\ell \le 2|X| + |Y| - 1 = |R| + |X| - 1$. Hence $|R|+|X| \ge \ell +1$ in all cases, as desired.
%If $H$ consists entirely of edges incident to $X$, then $\ell = e(H) \le 2|X| \le |R| +|X| - 1$ (as $Y \ne \emptyset$). Hence $|R|+|X| \ge \ell +1$.
%Finally, suppose that $H$ contains an edge between $X$ and $Y$. Then since $H$ is a path, there are at most $|Y|-1$ edges of $H$ internal to $Y$ and so $\ell \le 2|X| + |Y| - 1 = |R| + |X| - 1$. Hence $|R|+|X| \ge \ell +1$ in all cases, as desired.
\end{proof}

%%%%%%%%%%%%%%%%%%%%%%%%%%%%%%%%%%%%%%%%%%%%%%%%%%%%%%%%%%%%%%%%%%%%%%%%%%

\section{\label{sec:exactpathcalcs}Determining $\tilde{r}(P_{3},P_{\ell+1})$
for $\ell\ge2$}

Theorem~\ref{thm:lowerbound} implies that $\tilde{r}(P_{3},P_{\ell+1}) \ge \lceil5\ell/4\rceil$
for $\ell\ge2$.
To bound $\tilde{r}(P_{3},P_{\ell+1})$ above, we shall present a strategy
for Builder. In the discussion that follows, we assume for clarity
that Painter will never voluntarily lose the $\tilde{r}(P_{3},P_{\ell+1})$-game.

Builder will use the threat of a red $P_3$ to force a blue $P_{\ell+1}$.
First, Builder will use Lemma~\ref{lem:3-buildingblock} to construct a blue path $P$ with one endpoint incident to a red edge.
Builder will then use a procedure outlined in Lemma~\ref{lem:3-mainwork} to efficiently extend $P$ until it has length between $\ell - 4$ and $\ell$.
Finally, Builder will carefully extend $P$ into a blue $P_{\ell+1}$, yielding a tight upper bound for $\tilde{r}(P_3, P_{\ell+1})$ (see Theorem~\ref{thm:P3-Pl-exact}).

\begin{lem}
\label{lem:3-buildingblock}Let $q\in\N$ with $q\ge5$.
Builder can force one of the following structures independent of Painter's choices:
\begin{enumerate}
\item a red $P_{3}$ in at most $q-1$ rounds.
\item a blue $P_{q}$ in $q-1$ rounds.
\item a blue $P_{t}$ with one endpoint incident to a red edge
in $t$ rounds for some $4\le t\le q-1$.
\end{enumerate}
\end{lem}
\begin{proof}
Builder first chooses an arbitrary vertex $x_1$, then proceeds as follows.
Suppose that Builder has already obtained a blue path $x_1\dots x_i$ in $i-1$ rounds for some $1 \le i < q$. 
Builder then chooses the edge $x_ix_{i+1}$, where $x_{i+1}$ is a new vertex. If Painter colours $x_ix_{i+1}$ blue, we have obtained a 
blue path $x_1\dots x_{i+1}$ in $i$ rounds, and so if $i+1 < q$ we may repeat the process. If Painter colours 
all such edges blue, we will obtain a blue path $x_1\dots x_q$ in $q-1$ rounds and
achieve (ii). Suppose instead that for some $1 \le i \le q-1$, within $i$ rounds we obtain 
a path $x_1\dots x_{i+1}$ such that $x_1\dots x_i$ is blue and $x_ix_{i+1}$ is red.
If $i\ge4$ then we have achieved (iii), so suppose in addition $i\le3$.

First suppose $i\in\{1,2\}$. In this case, Builder chooses the two
edges $x_{i}v$ and $vx_{i+1}$ where $v$ is a new vertex. If $i=1$,
Builder also chooses the edge $x_{i+1}w$ where $w$ is a new vertex.
If Painter colours $x_{i}v$, $vx_{i+1}$ or $x_{i+1}w$ red, then
$x_{i+1}x_{i}v$, $vx_{i+1}x_{i}$ or $x_{i}x_{i+1}w$ respectively
is a red $P_{3}$ and we have achieved (i). Otherwise, we have achieved
(iii). Indeed, if $i=1$ then $x_{1}vx_{2}w$ is a blue $P_{4}$ constructed
in 4 rounds with $x_{1}$ incident to the red edge $x_{1}x_{2}$, and if $i=2$
then $x_{1}x_{2}vx_{3}$ is a blue $P_{4}$ constructed in 4 rounds
with $x_{3}$ incident to the red edge $x_{3}x_{2}$.

Finally, suppose $i=3$. Then Builder chooses the edge $x_{4}x_{1}$.
If Painter colours the edge red, then $x_{3}x_{4}x_{1}$ is a red
$P_{3}$ and we have achieved (i), so suppose Painter colours the
edge blue. Then $x_{4}x_{1}x_{2}x_{3}$ is a blue $P_{4}$ constructed
in 4 rounds with $x_{3}$ incident to the red edge $x_{3}x_{4}$, so we have achieved
(iii).\end{proof}

\begin{lem}
\label{lem:3-mainwork}Let $\ell\in\N$ with $\ell\ge4$. 
Builder can force one of the following structures independent of Painter's choices:
\begin{enumerate}
\item a red $P_{3}$ in at most $5\ell/4-1$ rounds.
\item a blue $P_{\ell+1}$ in at most $5\ell/4-1$ rounds.
\item a blue $P_{t}$ with one endpoint incident to a red edge
in at most $5t/4-1$ rounds for some $\ell-3\le t\le\ell$.
\end{enumerate}
\end{lem}

\begin{proof}
Throughout the proof, we assume for clarity that Painter will always avoid (i) and (ii) if possible. 
By Lemma~\ref{lem:3-buildingblock} (taking $q = \ell +1$) we may assume that Builder has constructed 
a blue $P_{t}$, say $v_1 \dots v_t$, which satisfies
\begin{enumerate}
\item[\rm ($\ast$)]  $v_1 \dots v_t$ has one endpoint incident to a red edge $v_1u$, and Builder constructed $v_1\dots v_t$ 
in at most $5t/4-1$ rounds. Moreover, $4 \le t\le\ell$.
\end{enumerate}
Note that $t \le 5t/4-1$ since $t \ge 4$.

If $t \ge \ell -3$, then we have achieved (iii).
Hence, we may assume that $4 \le t  < \ell -3$.
Without loss of generality, let $v_1u$ be a red edge as in ($\ast$).
Builder will extend $v_1\dots v_t$ as follows.
We apply Lemma~\ref{lem:3-buildingblock} with $q=\ell-t+1\ge5$ on a set of new vertices.
We split into cases depending on Painter's choice. 

\medskip\noindent \textbf{Case 1:} Builder obtains a red $P_{3}$ in at most $\ell-t$ rounds, as in Lemma~\ref{lem:3-buildingblock}(i).

In this case, Builder has spent at most $5t/4-1+\ell-t\le5\ell/4-2$ rounds
in total since $t \le \ell-4$, and so we have achieved (i).

\medskip\noindent \textbf{Case 2:} Builder obtains a blue path $w_{1}\dots w_{\ell-t+1}$ in 
$\ell-t$ rounds, as in Lemma~\ref{lem:3-buildingblock}(ii).

In this case, Builder has again spent at most $5\ell/4-2$ rounds in total.
Builder now chooses the edge $w_{1}v_{1}$. If Painter colours it
red, then $w_{1}v_{1}u$ is a red $P_{3}$ and we have achieved (i).
If Painter colours it blue, then $w_{\ell-t+1}\dots w_{1}v_{1}\dots v_{t}$
is a blue $P_{\ell+1}$ and we have achieved (ii).

\medskip\noindent \textbf{Case 3:} Builder obtains a blue path $w_{1}\dots w_{t'}$ and a red edge $w_{1}x$
in at most $t'$ rounds for some $4\le t'\le\ell-t$, as in Lemma~\ref{lem:3-buildingblock}(iii).

In this case, Builder has spent at most
\[
\frac{5t}{4}-1+t' = \frac{5t}{4}+\frac{5t'}{4} - \frac{t'}{4} - 1 \le \frac{5(t+t')}{4}-2\le\frac{5\ell}{4}-2
\]
 rounds in total. Builder now chooses the edge $v_{t}w_{1}$. If Painter
colours it red, then $v_{t}w_{1}x$ is a red $P_{3}$ and we have
achieved (i). 
If Painter colours it blue, then $v_{1}\dots v_{t}w_{1}\dots w_{t'}$ is a blue $P_{t+t'}$ with $v_1$ incident to the red edge $v_1u$.
Moreover, this $P_{t+t'}$ satisfies ($\ast$) with $t+t' >t$.
Hence by iterating the argument above, the result follows.\end{proof}

\begin{thm}
\label{thm:P3-Pl-exact}For all $\ell\ge2$, $\tilde{r}(P_{3},P_{\ell+1})=\lceil5\ell/4\rceil$.\end{thm}
\begin{proof}
Theorem~\ref{thm:lowerbound} implies that $\tilde{r}(P_{3},P_{\ell+1})\ge\lceil5\ell/4\rceil$.
It therefore suffices to prove that Builder 
can win the $\tilde{r}(P_3,P_{\ell+1})$-game within $\lceil5\ell/4\rceil$
rounds. First note that $\tilde{r}(P_3,P_3) = 3$ and $\tilde{r}(P_3,P_4) = 4$, as shown by Grytczuk, Kierstead and Pra\l{}at~\cite{grytczuk} and Pra\l{}at~\cite{pralat} respectively, 
so we may assume $\ell\ge4$. Applying
Lemma~\ref{lem:3-mainwork}, either Builder obtains a blue path $v_{1}\dots v_{t+1}$
and a red edge $v_{1}u$ in at most $5(t+1)/4-1$ rounds for some
$\ell-3\le t+1\le\ell$ or we are done. Write
\[
r(t)=\ceil{\frac{5\ell}{4}}-\left(\floor{\frac{5(t+1)}{4}}-1\right)=\ceil{\frac{\ell}{4}}-\floor{\frac{t+1}{4}}+(\ell-t),
\]
and note that Builder has at least $r(t)$ rounds left to construct either
a red $P_{3}$ or a blue $P_{\ell+1}$. We now split into cases depending
on the precise value of $t$.

\medskip\noindent \textbf{Case 1:} $t=\ell-1$, so that $r(t)=1$.

Builder chooses the edge $v_{0}v_{1}$, where $v_{0}$ is a new vertex.
If Painter colours it red, then $v_{0}v_{1}u$ is a red $P_{3}$ and
we are done. Otherwise, $v_{0}v_{1}\dots v_{\ell}$ is a blue $P_{\ell+1}$
and we are done.

\medskip\noindent \textbf{Case 2:} $t=\ell-2$, so that $r(t)\ge3$.

Builder chooses the edge $v_{\ell-1}x$, where $x$ is a new vertex.
If Painter colours it blue, then we are in Case 1 with an extra round
to spare. If Painter colours it red, Builder chooses the edges $v_{\ell-1}w$
and $wx$, where $w$ is a new vertex. If Painter colours either edge
red then $xv_{\ell-1}w$ or $wxv_{\ell-1}$ respectively is a red
$P_{3}$ and we are done. Otherwise, $v_{1}\dots v_{\ell-1}wx$ is
a blue $P_{\ell+1}$ and we are done.

\medskip\noindent \textbf{Case 3:} $t=\ell-3$, so that $r(t)\ge4$.

Builder chooses the edge $v_{\ell-2}x$, where $x$ is a new vertex.
If Painter colours it blue, then we are in Case 2. If Painter colours
it red, Builder chooses the edges $v_{\ell-2}w$, $wx$ and $xy$,
where $w$ and $y$ are new vertices. If Painter colours any of these
edges red then $xv_{\ell-2}w$, $wxv_{\ell-2}$ or $v_{\ell-2}xy$
respectively is a red $P_{3}$ and we are done. Otherwise, $v_{1}\dots v_{\ell-2}wxy$
is a blue $P_{\ell+1}$ and we are done.

\medskip\noindent \textbf{Case 4:} $t=\ell-4$, so that $r(t)\ge5$.

Builder chooses the edge $v_{\ell-3}x$, where $x$ is a new vertex.
If Painter colours it blue, then we are in Case 3. If Painter colours
it red, Builder chooses the edges $v_{0}v_{1}$, $v_{\ell-3}w$, $wx$
and $xy$, where $v_{0}$, $w$ and $y$ are new vertices. If Painter
colours any of these edges red then $v_{0}v_{1}u$, $xv_{\ell-3}w$,
$wxv_{\ell-3}$ or $v_{\ell-3}xy$ respectively is a red $P_{3}$
and we are done. Otherwise, $v_{0}v_{1}\dots v_{\ell-3}wxy$ is a
blue $P_{\ell+1}$ and we are done.
\end{proof}

\section{\label{sec:exactcyclecalcs}Determining $\tilde{r}(P_{3},C_{\ell})$
for $\ell\ge3$}

Our aim is to determine $\tilde{r}(P_{3},C_{\ell})$ for all $\ell \ge 3$, so proving Theorem~\ref{thm:pathresultsk=2}.
As a warmup, we first determine $\tilde{r}(P_{3},C_{3})$ and $\tilde{r}(P_{3},C_{4})$.
Note that Theorem~\ref{thm:lowerbound} implies that $\tilde{r}(P_{3},C_{3}) \ge 5 \ell /4$ for all $\ell \ge 3$, but this lower bound is too weak when $\ell \le 4$.
Instead, we consider the $\{C_{\ell}\}$-blocking strategy for Painter in an $\tilde{r}(C_{\ell},P_{3})$-game.\COMMENT{We do need to specify the game here -- otherwise the $\{C_\ell\}$-blocking strategy isn't uniquely defined. -JAL 30/09}
\begin{prop}
\label{prop:smallcyclelower}For all $\ell\ge3$, we have $\tilde{r}(P_{3},C_{\ell})\ge\ell+2$.\end{prop}
\begin{proof}
We consider the $\{C_{\ell}\}$-blocking strategy for Painter in the
$\tilde{r}(C_{\ell},P_{3})$-game. Let $R$ be an edge-minimal $\{C_{\ell}\}$-scaffolding
for $P_{3}$. Then $R$ must contain two distinct $P_{\ell}$'s, so
$e(R)\ge\ell$. The result therefore follows from
Proposition~\ref{prop:blocking}.
\end{proof}
The upper bounds are both relatively straightforward.
\begin{prop}
\label{prop:smallcycleexact}We have $\tilde{r}(P_{3},C_{3})=5$ and
$\tilde{r}(P_{3},C_{4})=6$.\end{prop}
\begin{proof}
By Proposition~\ref{prop:smallcyclelower}, we have $\tilde{r}(P_{3},C_{3})\ge5$
and $\tilde{r}(P_{3},C_{4})\ge6$. It is easy to show that $r(P_{3},C_{4})=4$ (see
e.g. Radziszowski~\cite{radziszowski}), so we also have $\tilde{r}(P_{3},C_{4})\le \binom{4}2 = 6$ as Builder may simply choose
the edges of a $K_{4}$. It therefore suffices to prove that Builder
can win the $\tilde{r}(P_{3},C_{3})$-game in 5 rounds.

Take new vertices $u$, $v$, $w$, $x$, $y$ and $z$. Builder first
chooses the edges $uv$, $uw$ and~$ux$. If Painter colours more
than one of these edges red, then we have obtained a red $P_{3}$
and we are done.

Suppose Painter colours $uv$, $uw$ and $ux$ blue. Then Builder
chooses the edges $vw$ and $wx$. If Painter colours either edge
blue, then $vwuv$ or $wxuw$ respectively is a blue $C_{3}$ and
we are done. If Painter colours both edges red, then $vwx$ is a red $P_3$ and we are done.

Finally, suppose Painter colours (without loss of generality) $uv$
red, but $uw$ and $ux$ blue. Then Builder chooses the edge $xy$.
If Painter colours $xy$ red, Builder chooses the edge $wx$, yielding
either a red $P_{3}$ (namely $wxy$), or a blue $C_{3}$, $wxuw$, and we
are done. If Painter colours $xy$ blue, Builder chooses the edge
$yu$, yielding either a red $P_{3}$ (namely $yuv$) or a blue $C_{3}$ (namely $uxyu$), and we are done.
\end{proof}

We now determine $\tilde{r}(P_{3},C_{\ell})$ for $\ell\ge5$. As
in Section~\ref{sec:exactpathcalcs}, Builder's strategy will be to build
up a long blue path using Lemma~\ref{lem:3-mainwork}.
Builder will then carefully close this path into a blue $C_{\ell}$.
\begin{thm}
\label{thm:P3-Cl-exact}For all $\ell\ge5$, $\tilde{r}(P_{3},C_{\ell})=\lceil5\ell/4\rceil$.\end{thm}
\begin{proof}
Theorem~\ref{thm:lowerbound} implies that $\tilde{r}(P_{3},C_{\ell})\ge\lceil5\ell/4\rceil$.
It therefore suffices to prove that Builder can 
win the $\tilde{r}(P_3,C_\ell)$-game within $\lceil5\ell/4\rceil$
rounds. By Lemma~\ref{lem:3-mainwork}, Builder can force one of the following structures independent of Painter's choices:
\begin{enumerate}
\item a red $P_{3}$ in at most $5(\ell-1)/4-1$ rounds.
\item a blue $P_{\ell}$ in at most $5(\ell-1)/4-1$ rounds.
\item a blue $P_{t}$ with one endpoint incident to a red edge
in at most $5t/4-1$ rounds for some $\ell-4\le t\le\ell-1$.
\end{enumerate}
\noindent If Painter chooses (i), then we are done. Suppose Painter
chooses (ii), so that Builder has at least
\[
\ceil{\frac{5\ell}{4}}-\left(\frac{5(\ell-1)}{4}-1\right)=\ceil{\frac{5\ell}{4}}-\frac{5\ell}{4}+\frac{9}{4}>2
\]
rounds to construct a red $P_{3}$ or a blue $C_{\ell}$, and let $v_{1}\dots v_{\ell}$
be the corresponding blue path. Then Builder chooses the edges $v_{\ell}v_{1}$,
$v_{1}v_{3}$ and $v_{\ell}v_{2}$. If Painter colours $v_{\ell}v_{1}$
blue then $v_{1}\dots v_{\ell}v_{1}$ is a blue $C_{\ell}$ and we
are done. If Painter colours $v_{\ell}v_{1}$ red and $v_{1}v_{3}$
or $v_{\ell}v_{2}$ red, then $v_{\ell}v_{1}v_{3}$ or $v_{1}v_{\ell}v_{2}$
respectively is a red $P_{3}$ and we are done. Finally, if Painter
colours both $v_{1}v_{3}$ and $v_{\ell}v_{2}$ blue, then $v_{1}v_{3}v_{4}\dots v_{\ell}v_{2}v_{1}$
is a blue $C_{\ell}$ and we are done. 

Finally, suppose Painter chooses (iii). Let $v_{1}\dots v_{t}$ be the corresponding blue path and let 
$v_{1}u$ be a red edge. Write 
\[
r(t)=\ceil{\frac{5\ell}{4}}-\left(\floor{\frac{5t}{4}}-1\right)=\ceil{\frac{\ell}{4}}-\floor{\frac{t}{4}}+\ell-t+1,
\]
so that Builder has at least $r(t)$ rounds left to construct either
a red $P_{3}$ or a blue~$C_{\ell}$. We split into cases depending
on the precise value of $t$.

\medskip\noindent \textbf{Case 1:} $t=\ell-1$, so that $r(t)\ge3$.

Builder first chooses the edge $v_{\ell-1}w$, where $w$ is a new
vertex. If Painter colours $v_{\ell-1}w$ blue, then Builder chooses
the edge $wv_{1}$. If Painter colours $wv_{1}$ red then $wv_{1}u$
is a red $P_{3}$, and if Painter colours $wv_{1}$ blue then $v_{1}v_{2}\dots v_{\ell-1}wv_{1}$
is a blue $C_{\ell}$. Now suppose Painter colours $v_{\ell-1}w$
red instead. Then Builder chooses the edges $v_{\ell-1}x$ and $xv_{1}$,
where $x$ is a new vertex. If Painter colours either edge red, then
$wv_{\ell-1}x$ or $xv_{1}u$ respectively is a red $P_{3}$ and we
are done. Otherwise, $v_{1}\dots v_{\ell-1}xv_{1}$ is a blue $C_{\ell}$
and we are done.

\medskip\noindent \textbf{Case 2:} $t=\ell-2$, so that  $r(t)\ge4$.

Builder first chooses the edge $v_{\ell-2}w$, where $w$ is a new
vertex. If Painter colours $v_{\ell-2}w$ blue then we are in Case
1, so suppose Painter colours $v_{\ell-2}w$ red. Builder then chooses
the edges $v_{\ell-2}x$, $xw$ and $wv_{1}$, where $x$ is a new
vertex. If Painter colours any of these edges red, then $wv_{\ell-2}x$,
$xwv_{\ell-2}$ or $v_{\ell-2}wv_{1}$ respectively is a red $P_{3}$
and we are done. Otherwise, $v_{1}v_{2}\dots v_{\ell-2}xwv_{1}$ is
a blue $C_{\ell}$ and we are done.

\medskip\noindent \textbf{Case 3:} $t=\ell-3$, so that $r(t)\ge5$.

Builder first chooses the edge $v_{\ell-3}w$, where $w$ is a new
vertex. If Painter colours $v_{\ell-3}w$ blue then we are in Case
2, so suppose Painter colours $v_{\ell-3}w$ red. Builder then chooses
the edges $v_{\ell-3}x$, $xw$, $wy$ and $yv_{1}$, where $x$ and
$y$ are new vertices. If Painter colours any of these edges red,
then $wv_{\ell-3}x$, $xwv_{\ell-3}$, $v_{\ell-3}wy$ or $yv_{1}u$
respectively is a red $P_{3}$ and we are done. Otherwise, $v_{1}v_{2}\dots v_{\ell-3}xwyv_{1}$
is a blue $C_{\ell}$ and we are done.

\medskip\noindent \textbf{Case 4:} $t=\ell-4$, so that $r(t)\ge6$.

Builder first chooses two edges $wx$ and $xy$, where $w$, $x$
and $y$ are new vertices. If Painter colours both edges red, $wxy$
is a red $P_{3}$ and we are done. Now suppose that Painter colours
one edge blue and one red, say $wx$ red and $xy$ blue. Then Builder
chooses the edges $v_{\ell-4}w$, $wz$, $zx$ and $yv_{1}$, 
where $z$ is a new vertex. If Painter
colours any of these edges red, then $v_{\ell-4}wx$, $xwz$, $zxw$
or $yv_{1}u$ respectively is a red $P_{3}$ and we are done. Otherwise,
$v_{1}v_{2}\dots v_{\ell-4}wzxyv_{1}$ is a blue $C_{\ell}$ and we
are done.

We may therefore assume that Painter colours both $wx$ and $xy$
blue. Builder now chooses the edge $v_{\ell-4}w$. If Painter colours
$v_{\ell-4}w$ blue, we are in Case 1 (taking our path to be $v_{1}v_{2}\dots v_{\ell-4}wxy$),
so suppose Painter colours $v_{\ell-4}w$ red. Then Builder chooses
the edges $v_{\ell-4}z$, $zw$ and $yv_{1}$, where $z$ is a new
vertex. If Painter colours any of these edges red, then $wv_{\ell-4}z$,
$zwv_{\ell-4}$ or $yv_{1}u$ respectively is a red $P_{3}$ and we
are done. Otherwise, $v_{1}v_{2}\dots v_{\ell-4}zwxyv_{1}$ is a blue
$C_{\ell}$ and we are done.
\end{proof}

\section{\label{sec:C4}Bounding $\tilde{r}(C_{4},P_{\ell+1})$ for $\ell \ge 3$}

Our aim is to prove Theorem~\ref{prop:C4}, i.e. to bound $\tilde{r}(C_{4},P_{\ell+1})$ for all $\ell \ge 3$.
%The lower bound is proved by considering a $\{C_{4}\}$-blocking strategy for Painter.
%
%\begin{prop} \label{prop:Cklower}
%Let $k \in \mathbb{N}$ with $k \ge 3$.
%Let $H$ be a connected graph.
%Then $\tilde{r}( C_{k}, H ) \ge |H| + e(H) -1$.
%\end{prop}
%
%
%\begin{proof}
%We consider the $\{C_{k}\}$-blocking strategy for Painter in the $\tilde{r}(C_{k},H)$-game.
%Let $R$ be a $\{ C_k \}$-scaffolding for~$H$ with $e(R)$ minimal.
%Note that each $(R,\{C_{k}\})$-forceable edge must lie entirely in a component of~$R$.
%Since $H$ is connected, $R$ is connected and $|R| \ge |H|$.
%Hence, $e(R) \ge  |H| - 1$. 
%We are done by Proposition~\ref{prop:blocking}.
%\end{proof}
First we prove that $\tilde{r}( C_{4},P_{4}) = 8$.
%Note that a more detailed analysis of the $\{C_{4}\}$-blocking strategy for Painter is needed in order to obtain a better lower bound.

\begin{prop} \label{prop:C4P4}
$\tilde{r}( C_{4},P_{4}) = 8$.
\end{prop}

\begin{proof}
First, we consider the $\{C_{4}\}$-blocking strategy for Painter in the $\tilde{r}(C_{4},P_{4})$-game.
Let $R$ be an edge-minimal $\{C_{4}\}$-scaffolding for~$P_{4}$.
Then $R$ must contain three distinct $P_{4}$'s, so $e(R)\ge 5$ as $R$ is $C_4$-free.
Proposition~\ref{prop:blocking} implies that $\tilde{r}( C_{4},P_{4}) \ge 8$.

It therefore suffices to prove that Builder can win the $\tilde{r}(C_4,P_4)$-game within $8$ rounds.
Builder first chooses the edges $uv_1$, \dots, $uv_4$ for distinct vertices $u,v_1, \dots, v_4$.
Without loss of generality we may assume that there exists an integer $j$ such that Painter colours the edges $uv_i$ blue if $i \le j$, and red otherwise.

Suppose $j \ge 2$.
Then Builder chooses four edges $v_1w$, $v_2w$, $v_1w'$ and $v_2w'$, where $w$ and $w'$ are new vertices.
If Painter colours all edges red, then $v_1 w v_2 w' v_1$ is a red~$C_4$.
If Painter colours one of the edges blue say $v_2w$, then $ v_1 u v_2 w$ is a blue~$P_4$.

Suppose $j \le 1$.
Then Builder chooses edges $v_1v_2$ and $v_1v_3$.
If Painter colours both edges red, then $u v_2 v_1 v_3 u$ is a red~$C_4$.
Suppose that Painter colours both edges blue.
Builder then chooses the edges $v_2v_4$ and $v_3v_4$.
If Painter colours both $v_2v_4$ and $v_3v_4$ red, then $u v_2 v_4 v_3  u$ is a red $C_4$.
Otherwise, $v_3 v_1 v_2 v_4$ or $v_2 v_1 v_3 v_4$ is a blue $P_4$.
Therefore we may assume that $v_1v_2$ is blue and $v_1v_3$ is red.
Further suppose that $j=1$ and so $uv_1$ is blue.
Then Builder chooses the edges $v_2v_3$ and $v_2v_4$.
If Painter colours one of them blue, then $u v_1 v_2 v_3$ or $u v_1 v_2 v_4$ is a blue~$P_4$.
Otherwise $u v_3 v_2 v_4 u$ is a red~$C_4$.
Finally, suppose that $j = 0$.
Builder chooses the edges $v_2v_3$ and $v_3 v_4$.
If Painter colours one of them red, then $u v_1 v_3 v_2 u$ or $u v_1 v_3 v_4 u$ is a red~$C_4$.
Otherwise $v_1 v_2 v_3 v_4$ is a blue $P_4$.
\end{proof}

We now prove Theorem~\ref{prop:C4}.

\begin{proof}[Proof of Theorem~\ref{prop:C4}]
The lower bound follows from Lemma~\ref{lem:Ckscaf} and\break 
$\tilde{r}(C_{4},P_{4})=8$ by Proposition~\ref{prop:C4P4}.
To prove the theorem, it is enough to show that $\tilde{r}( C_{4},P_{\ell +1}) \le  4 \ell -4$ for all $\ell \ge 3$.
We proceed by induction on~$\ell$.
By Proposition~\ref{prop:C4P4}, this is true for $\ell = 3$.
Suppose instead that $\ell \ge 4$ and 
Builder first spends at most $4 \ell - 8$ rounds forcing Painter to
construct a red $C_4$ or a blue $P_{\ell} = v_1 \dots v_{\ell}$.
(This is possible by the induction hypothesis.)
We may assume that the latter holds or else we are done.
Then Builder chooses four edges $v_1 x$, $v_{\ell} x$, $v_1 y$ and $v_{\ell} y$, where $x$ and $y$ are new vertices.
If Painter colours all edges red, then $v_1 x v_{\ell} y v_1$ is a red~$C_4$.
If Painter colours one of the edges blue, say $v_{\ell} x$, then $v_1 \dots v_{\ell}  x$ is a blue~$P_{\ell+1}$.
In total, Builder has chosen at most $4 \ell -4$ edges and the proposition follows.
\end{proof}

\section*{Acknowledgement}
We would like to thank the referee for pointing out an error in the early manuscript.

\appendix

\section{\label{sec:P4-bound}Bounding $\tilde{r}(P_{4},P_{\ell+1})$
for $\ell\ge3$}

Here, we prove Theorem~\ref{thm:pathresultsk=3}.
Lemma~\ref{lem:p4forestscaffold} implies that $\tilde{r}(P_4,P_{\ell+1}) \ge (7 \ell +2)/5$ for $\ell \ge 3$.
It therefore suffices to bound $\tilde{r}(P_{4},P_{\ell+1})$ above, which we do in Theorem~\ref{thm:P4-upper}.
In the following discussion we take on the role of Builder, and we will assume for clarity
that Painter will not voluntarily lose the game by creating a red~$P_{4}$.

We will employ the following strategy to construct a blue $P_{\ell+1}$.
We will obtain two (initially
trivial) vertex-disjoint blue paths $Q$ and $R$, repeatedly extend them, and then join them together to form
a blue $P_{\ell+1}$ when they are sufficiently long.
Here $Q$ is distinct from $R$ in that we require one of $Q$'s endpoints to be incident to a red edge $bc$ disjoint from $V(R)$.
Some of our methods for extending a blue path require this property, and others destroy it. 
Thus at each stage we will extend either $Q$ or $R$ depending on which of our extension methods Painter allows us to use.

We will use the following lemma to join $Q$ and $R$ together (and sometimes to extend $Q$).%
	
\begin{lem}
\label{lem:joinpaths}
Let $Q$ be a (possibly trivial) blue path with endpoints $a$ and $b$, where $b$ is incident to a red edge~$bc$.
Let $R$ be a (possibly trivial) blue path vertex-disjoint from $V(Q) \cup \{c\}$.
Then Builder can force Painter to construct one of the following while uncovering at most $2$ edges:
\begin{enumerate}
\item a blue path $Q'$ of length $e(Q)+e(R)+1$ with one endpoint incident to a red edge.
\item a red $P_{4}$.
\end{enumerate}
\end{lem}

\begin{proof}
First suppose that $R$ is non-trivial, and let $x$ and $y$ be the endpoints of~$R$.
Moreover, suppose that either $a = c$ or $Q$ is trivial, so that both endpoints of $Q$ are incident to~$bc$.
Builder chooses the edges $bx$ and~$cy$.
If Painter colours both edges red, then $xbcy$ is a red~$P_4$.
Hence, without loss of generality, we may assume that Painter colours $bx$ blue.
Then $Q' := aQbxRy$ is a blue path of length $e(Q)+e(R)+1$, where $a$ is incident to the red edge~$bc$.

Now suppose that $Q$ is non-trivial and $a \ne c$.
Builder chooses the edge $ax$.
If Painter colours $ax$ blue, then $bQaxRy$ is a blue path of length $e(Q)+e(R)+1$ with endpoint $b$ incident to the red edge~$bc$.
So we may assume that Painter colours $ax$ red.
Builder then chooses the edge $bx$.
If Painter colours $bx$ red, then $cbxa$ is a red~$P_4$.
Otherwise $Q' := aQbxRy$ is a blue path of length $e(Q)+e(R)+1$ where $a$ is incident to the red edge $ax$.

Finally, suppose $R$ is trivial with endpoint $x$. Let $y$ be a new vertex. Then the argument above implies the lemma on replacing $xRy$ with $x$ throughout.
\end{proof}

The arguments that follow are by necessity somewhat technical. The reader may therefore find the following intuition useful. 
\begin{enumerate}
\item For every seven edges we uncover, we will extend either $Q$ or $R$ by five blue edges.
\item When we join $Q$ and $R$, $e(Q)+e(R)+1$ should not be too much greater than $\ell$. 
\end{enumerate}
It is clear that following the above principles will yield a bound of the form \linebreak $\tilde{r}(P_4,P_{\ell+1})\le 7\ell/5+C$ for some constant $C$. We will violate (i) in the first and last phases of Builder's strategy, but this introduces only constant overhead.

Before we can apply Lemma~\ref{lem:joinpaths} to join $Q$ and $R$ and obtain a blue $P_{\ell+1}$, 
we must extend them until $e(Q)+e(R)+1\ge\ell$.
Each time we extend $Q$ and $R$, we require two independent edges of the same
colour. (Naturally, we can obtain these by choosing three independent edges.) 
If these edges are blue, we may extend $Q$ efficiently using Lemma~\ref{lem:useblueedges} (see Section~\ref{sub:blue}).
If they are red, we may extend either $Q$ or $R$ efficiently using Lemma~\ref{lem:userededges} (see Section~\ref{sub:red}).
Note that the latter case is significantly harder.
We then apply Lemmas~\ref{lem:useblueedges} and~\ref{lem:userededges} repeatedly to prove Theorem~\ref{thm:P4-upper} (see Section~\ref{sub:fullalgo}).

In our figures throughout the section, we shall represent blue edges with solid lines and red edges with dotted lines.

\subsection{\label{sub:blue} Extending $Q$ using two independent blue edges $e$ and $f$.}

Throughout this subsection, $e$ and $f$ will be two independent blue edges vertex-disjoint from $Q$ and $R$.
We will prove that we can use these two edges to efficiently extend $Q$ -- see Lemma~\ref{lem:useblueedges}. We first 
define a special type of path which will be important to the extension process.

\begin{defn}\label{def:typeA}
We say that a path $xySz$ is \emph{of type~A} if $xy$ is a red edge and $S$ is a non-trivial blue path with endpoints $y$ and~$z$.
\end{defn}
Note that the above definition requires $x \notin V(S)$. For the remainder of the section, if we refer to a path $xySz$ of type A, we shall take it as read that $x,y,z$ and $S$ are as in Definition~\ref{def:typeA}.

We now sketch the proof of Lemma~\ref{lem:useblueedges}. By greedily extending the blue edge~$e$ into a path, Builder can obtain either a long blue path or a path of type~A (see Lemma~\ref{lem:findA-path}).
If Builder obtains a long blue path $P$, then we can simply join $P$ and $Q$ together using Lemma~\ref{lem:joinpaths}.
Suppose instead Builder obtains a path $xySz$ of type~A.
Then we use Lemma~\ref{lem:useA-path} to efficiently join $S$ and $Q$ together.
In either case, the resulting blue path $Q'$ also has an endpoint incident to a red edge, so $Q'$ retains the defining property of~$Q$.

We first prove that Builder can obtain either a long blue path or a path of type~A by greedily extending~$e$.
\begin{lem}
\label{lem:findA-path}
Let $m\in\N$ and let $e$ be a blue edge.
Then Builder can force Painter to construct one of the following:
\begin{enumerate}
\item a path $xySz$ of type~A with $e(S) = t$ while uncovering $t$ edges for some $1 \le t <  m$.
\item a blue path of length $m$ while uncovering $m-1$ edges.
\end{enumerate}
\end{lem}

\begin{proof}
Let $S_{1}$ be the blue path formed by $e$. Builder proceeds to
extend $S_{1}$ greedily until either Builder has constructed a blue path of length
$m$ or Painter has coloured an edge red. 

Indeed, suppose $S_{i}$ is a blue path of length $i$ for some $1\le i\le m-1$
with endpoints $y$ and $z$, and that Builder has uncovered $i-1$
edges in forming $S_{i}$ from~$S_{1}$. Then Builder chooses the
edge $xy$, where $x$ is a new vertex. 
If Painter colours $xy$ red then $xyS_iz$ is a path of type~A with $e(S_i) = i$, where $1\le i< m$. Moreover, Builder has uncovered $i$ edges in constructing it, and so we have achieved (i).
If instead Painter colours $xy$ blue, then $S_{i+1} := xyS_iz$ is a blue path of length $i+1$ and Builder has uncovered $i$ edges in constructing it.

By repeating this process, Builder must either obtain a path of type
A as in (i) or a blue path $S_{m}$ of length $m$ as in (ii).
\end{proof}

We now prove that Builder can use a path of type~A to efficiently
extend~$Q$.
Recall that we were given two independent blue edges, $e$ and $f$, and that we have already used $e$ to construct a path of type~A.

\begin{lem}\label{lem:useA-path}
Suppose $Q$ is a non-trivial blue path with endpoints $a$ and $b$, where $b$ is incident to a red edge $bc$.%
\COMMENT{We should be OK in the case $a=c$.}
Suppose $xySz$ is a path of type~A which is vertex-disjoint from $V(Q)\cup\{c\}$.
Further suppose that $f = vw$ is a blue edge vertex-disjoint from $V(Q)\cup V(xySz) \cup \{c\}$.
Then Builder can force Painter to construct one of the following:
\begin{enumerate}
\item a blue path $Q'$ of length $e(Q)+e(S)+2$ with one endpoint $b'$ incident to a red edge $b'c'$ while uncovering $2$ edges.
Moreover, $f$ is vertex-disjoint from $V(Q')\cup \{c'\}$.

\item a blue path $Q'$ of length $e(Q)+e(S)+4$ with one endpoint
incident to a red edge $b'c'$ while uncovering $4$ edges. (Note that $f$ need not be vertex-disjoint from $V(Q')\cup \{c'\}$.)

\item a red $P_{4}$ while uncovering at most $4$ edges.
\end{enumerate}
\end{lem}
\begin{proof}

Builder chooses the edge $ax$.
First suppose Painter colours $ax$ blue. Builder then chooses the
edge $by$. If Painter colours the edge $by$ red, then $cbyx$ is a red $P_{3}$
and we have achieved (iii). Suppose not. Then $Q': = xaQbySz$ (see Figure~\ref{fig:useA-path1}) is a blue
path of length $e(Q)+e(S)+2$, where $x$ is incident to the red edge
$xy$, and we have achieved (i).

Now suppose Painter instead colours $ax$ red. Builder then chooses
the edges $av$, $wy$ and $xb$. If Painter colours any of these
edges red, then $yxav$, $wyxa$ or $yxbc$ respectively is a red
$P_{4}$ and we have achieved (iii). Suppose not. Then $Q': = xbQavwySz$ (see Figure~\ref{fig:useA-path2})
is a blue path of length $e(Q)+e(S)+4$, where $x$ is incident to
the red edge $xy$, and we have achieved (ii). 
\end{proof}

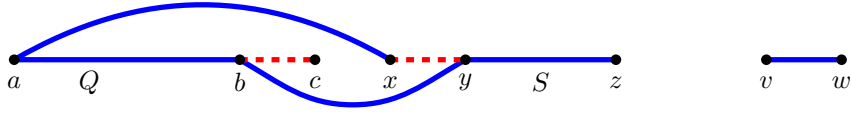
\begin{figure}[tp]
\centering
\begin{tikzpicture}
			\begin{scope}[line width=2pt, blue]
			\draw (7,0) to (9,0);
			\draw (1,0) to (4,0);
			\draw (11,0) to (12,0);

			\draw (1,0) to [bend left] (6,0); %ax			
			\draw (4,0) to [out=330, in=180] (5.5, -0.6) to [out=0, in=210] (7,0); %by
			\end{scope}
			
			\begin{scope}[line width=2pt, red, dashed]
			\draw (6,0) to (7,0); %xy
			\draw (4,0) to (5,0); %bc
			\end{scope}

			\fill (5,0) circle (2pt);
			\fill (1,0) circle (2pt);
			\fill (4,0) circle (2pt);
			
			\fill (6,0) circle (2pt);
			\fill (7,0) circle (2pt);
			\fill (9,0) circle (2pt);
			
			\fill (11,0) circle (2pt);
			\fill (12,0) circle (2pt);
		
			\node at (5,-0.3) {$c$};
			\node at (1,-0.3) {$a$};
			\node at (4,-0.3) {$b$};
			\node at (2,-0.3) {$Q$};
	
			\node at (6,-0.3) {$x$};
			\node at (7,-0.3) {$y$};
			\node at (8,-0.3) {$S$};
			\node at (9,-0.3) {$z$};
			
			\node at (11,-0.3) {$v$};
			\node at (12,-0.3) {$w$};
\end{tikzpicture}
\caption{Extending $Q$ using a path of type~A as in Lemma~\ref{lem:useA-path}(i).}
\label{fig:useA-path1}
\end{figure}

\begin{figure}[tp]
\centering
\begin{tikzpicture}
			\begin{scope}[line width=2pt, blue]
			\draw (7,0) to (9,0);%ySz
			\draw (2,0) to (5,0);%bQa
			\draw (5.5,1) to (6.5,1);%vw

			\draw (5,0) to (5.5,1); %av
			\draw (6.5,1) to (7,0); %wy
			\draw (2,0) to [out=330, in=180] (4,-0.6) to [out=0, in=210] (6,0); %bx
			\end{scope}
			
			\begin{scope}[line width=2pt, red, dashed]
			\draw (5,0) to (7,0); %axy
			\draw (1,0) to (2,0); %cb
			\end{scope}

			\fill (1,0) circle (2pt);
			\fill (2,0) circle (2pt);
			\fill (5,0) circle (2pt);
			
			\fill (6,0) circle (2pt);
			\fill (7,0) circle (2pt);
			\fill (9,0) circle (2pt);
			
			\fill (5.5,1) circle (2pt);
			\fill (6.5,1) circle (2pt);
		
			\node at (1,-0.3) {$c$};
			\node at (5,-0.3) {$a$};
			\node at (2,-0.3) {$b$};
			\node at (3.5,-0.3) {$Q$};
	
			\node at (6,-0.3) {$x$};
			\node at (7,-0.3) {$y$};
			\node at (8,-0.3) {$S$};
			\node at (9,-0.3) {$z$};
			
			\node at (5.5,1.3) {$v$};
			\node at (6.5,1.3) {$w$};
\end{tikzpicture}
\caption{Extending $Q$ using a path of type~A and an blue independent edge $vw$ as in Lemma~\ref{lem:useA-path}(ii).}
\label{fig:useA-path2}
\end{figure}
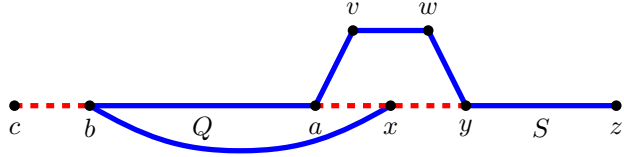

We now consolidate Lemmas~\ref{lem:findA-path} and~\ref{lem:useA-path} into a single lemma which says that given
two independent blue edges, Builder can efficiently extend $Q$. 
In applying Lemma~\ref{lem:useblueedges}, we will take $m$ to be $\ell - e(Q) - e(R) - 1$. 
Thus if we can extend $Q$ by at least $m$ edges, then we can join $Q$ and $R$ to obtain a blue $P_{\ell+1}$ immediately afterwards.

\begin{lem}
\label{lem:useblueedges}
Let $m \in \mathbb{N}$.
Suppose $Q$ is a non-trivial blue path with endpoints~$a$ and~$b$, where $b$ is incident to a red edge $bc$.
Suppose $e$ and $f$ are two independent blue edges which are vertex-disjoint from $V(Q)\cup\{c\}$.
Then Builder can force Painter to construct one of the following:
\begin{enumerate}
\item a blue path $Q'$ with $e(Q') = e(Q) + \ell'$ for some $3 \le \ell' \le m+3$ such that $Q'$ has an endpoint $b'$ incident to a red edge $b'c'$. 
A total of $\ell'$ edges are uncovered in the process.
Moreover, if $\ell' < 5 \le m$,
then $f$ is vertex-disjoint from $V(Q') \cup \{c'\}$.

\item a red $P_4$ while uncovering at most $m+3$ edges.
\end{enumerate}
\end{lem}

\begin{proof}
We apply Lemma~\ref{lem:findA-path} to $e$ and $m$, and split into
cases depending on Painter's choice.

\medskip\noindent \textbf{Case 1:} As in Lemma~\ref{lem:findA-path}(i),
we obtain a path $xySz$ of type~A with $e(S) = t$ for some $1 \le t < m$ which
is vertex-disjoint from $V(f)\cup V(Q)\cup \{c\}$, while uncovering $t$ edges.

We apply Lemma~\ref{lem:useA-path} to $Q$, $xySz$ and $f$.
First suppose that as in Lemma~\ref{lem:useA-path}(i), we obtain a blue path $Q'$ of length $e(Q)+t+2$ with one endpoint incident to a red edge while preserving $f$'s independence.
In total we have uncovered $t+2$ edges.
Hence $Q'$ satisfies (i) on setting $\ell'=t+2$. 

Now suppose that as in Lemma~\ref{lem:useA-path}(ii), we obtain a blue path $Q'$ of length $e(Q)+t+4$ with one endpoint incident to a red edge.
We have uncovered $t+4$ edges in total.
Hence setting $\ell'=t+4$, we have achieved~(i) with $\ell' \ge 5$.

Finally, suppose that as in Lemma~\ref{lem:useA-path}(iii) we obtain a red $P_4$. Then we have uncovered at most $t+4 \le m+3$ edges in total and so we have achieved (ii).

\medskip\noindent \textbf{Case 2:} As in Lemma~\ref{lem:findA-path}(ii),
we obtain a blue path $S$ of length $m$ which is vertex-disjoint
from $V(Q)\cup \{c\}$ while uncovering $m-1$ edges.

We apply Lemma~\ref{lem:joinpaths} to $Q$ and $S$ to construct either a blue path $Q'$ of length $e(Q)+m+1$ with one endpoint incident to a red edge or a red $P_{4}$ while
uncovering at most 2 additional edges.
We have uncovered at most $m+1$
edges in total. Hence in the former case we have achieved~(i), and
in the latter case we have achieved~(ii). 
\end{proof}

%%%%%%%%%%%%%%%%%%%%%%%%%%%%%%%%%%%%%%%%%%%%%%%%%%%%%%%%%%%%%%%%%%%%%%%%%%%%%%%%%%%%%%%%
\subsection{\label{sub:red} Extending $Q$ and $R$ using two red edges $e$ and $f$.}

In this subsection, our aim is to extend $Q$ or $R$ efficiently when given two independent red edges $e$ and~$f$ -- see Lemma~\ref{lem:userededges}.
As in Section~\ref{sub:blue}, it will be convenient to define some special paths that we will use in the extension process. 
These paths can be viewed as analogues of paths of type~A.

\begin{defn}\label{def:typeB}
A path $vwxyz$ is \emph{of type~B} if $vw$ and $yz$ are red edges, and $wx$ and $xy$ are blue edges.
\end{defn}
\begin{defn}\label{def:typeC}
A path $T_{1}\ldots T_{k}$ is \emph{of type~C} if the following statements hold:
\begin{enumerate}
\item[(C1)] $k$ is odd and $k\ge3$.
\item[(C2)] $T_{1}$ is either a blue edge or a path of the form $x_1y_1z_1$, where
$z_1\in V(T_{2})$ and $y_1z_1$ is red (and $x_1y_1$ may be red or blue).
\item[(C3)] $T_{k}$ is either a blue edge or a path of the form $x_k y_k z_k$, where
$x_k\in V(T_{k-1})$ and $x_k y_k$ is red (and $y_k z_k$ may be red or blue).
\item[(C4)] $T_{2},T_{4},\ldots,T_{k-1}$ are blue paths. Exactly
one of these paths has length~$1$ and the rest have length~$2$.
\item[(C5)] $T_{3},T_{5},\ldots,T_{k-2}$ are all red $P_{3}$'s.
\end{enumerate}
We say $T_{1}\dots T_{k}$ is \emph{incomplete} if $T_1$ or $T_k$ is a red~$P_3$.
Otherwise, we say $T_{1}\dots T_{k}$ is \emph{complete}.
\end{defn}
For the remainder of the section, if we refer to a path $vwxyz$ of type~B or a path $T_1\dots T_k$ of type~C, we shall take it as read that $v,w,x,y,z$ and $T_1,\dots,T_k$ are as in Definitions~\ref{def:typeB} and~\ref{def:typeC} respectively. Note that paths of type~C are well-defined with respect to direction of traversal -- if $v_1 \dots v_p$ is a path of type
C, then so is $v_p \dots v_1$.%
\COMMENT{AL: had `$T_{1}\dots T_{k}$ is a path of type~C, then so is $T_{k}\dots T_{1}$' previously, but $T_{k}\dots T_{1}$ is not the reverse of $T_{1}\dots T_{k}$.}
See Figure~\ref{fig:typeC} for an example of a path of type~C.
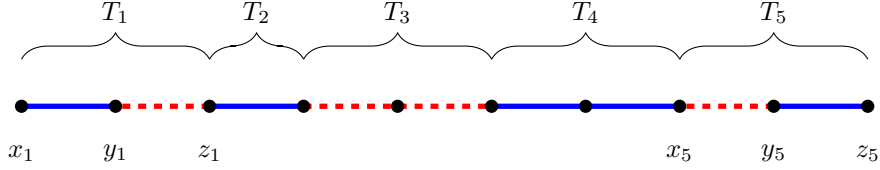
\begin{figure}[tp]
\centering
\begin{tikzpicture}[scale=1.25]
			\begin{scope}[line width=2pt, blue]
			\draw (3,0) to (4,0);
			\draw (6,0) to (8,0);
			\draw (1,0) to (2,0);
			\draw (9,0) to (10,0);
			\end{scope}
			
			\begin{scope}[line width=2pt, red, dashed]
			\draw (4,0) to (6,0);
			\draw (3,0) to (2,0);
			\draw (9,0) to (8,0);
			\end{scope}
			
			\begin{scope}[start chain]
			\foreach \i in {1,2,...,10}
			\fill (\i,0) circle (2pt);
			\end{scope}
			
			\node at (1,-0.5) {$x_1$};
			\node at (2,-0.5) {$y_1$};
			\node at (3,-0.5) {$z_1$};
			
			\node at (8,-0.5) {$x_5$};
			\node at (9,-0.5) {$y_5$};
			\node at (10,-0.5) {$z_5$};
			
			\draw [decorate,decoration={brace,amplitude=10pt}] (1,0.5) -- (3,0.5) node [above,midway,yshift=10pt] {$T_1$};
			\draw [decorate,decoration={brace,amplitude=10pt}] (3,0.5) -- (4,0.5) node [above,midway,yshift=10pt] {$T_2$};
			\draw [decorate,decoration={brace,amplitude=10pt}] (4,0.5) -- (6,0.5) node [above,midway,yshift=10pt] {$T_3$};
			\draw [decorate,decoration={brace,amplitude=10pt}] (6,0.5) -- (8,0.5) node [above,midway,yshift=10pt] {$T_4$};
			\draw [decorate,decoration={brace,amplitude=10pt}] (8,0.5) -- (10,0.5) node [above,midway,yshift=10pt] {$T_5$};
\end{tikzpicture}
\caption{A complete path $T_{1}\ldots T_{5}$ of type~C.}
\label{fig:typeC}
\end{figure}

We now sketch the proof of Lemma~\ref{lem:userededges}. Let $e$ and $f$ be two independent red edges. Using these edges, Builder can force either a path of type~B or a path of type~C using Lemma~\ref{lem:findBC-paths}. If Builder obtains a path $vwxyz$ of type~B, they will apply Lemma~\ref{lem:useB-path} to efficiently extend $Q$ using $vwxyz$.

Suppose instead Builder obtains a path $T_1\dots T_k$ of type~C. Then we run into a problem -- $T_1\dots T_k$ is not complete, and only a complete path of type~C may be used to efficiently extend $R$ (see Lemma~\ref{lem:useD-path}). Builder will therefore use Corollary~\ref{cor:extendC-path} to extend $T_1\dots T_k$ into a path $T_1'\dots T'_{k'}$ of type~C which is either complete or arbitrarily long. Builder then uses Lemma~\ref{lem:useD-path} to extend $R$ using $T'_1\dots T'_{k'}$. If $T_1'\dots T'_{k'}$ is complete, this extension is efficient; otherwise, Builder wins the game immediately afterwards by joining $Q$ and the resulting blue path. Thus an incomplete path of type~C is used to extend $R$ at most once over the course of the game, adding only constantly many rounds to the game's length.\COMMENT{These two paragraphs are the rephrasing we discussed. Feel free to revert if you prefer the old version. -JAL 15/10}

We first prove that given two independent
red edges Builder can force either a path of type~B or a path of type~C. 

\begin{lem}
\label{lem:findBC-paths}Given two independent red edges $e$ and $f$,
Builder can force Painter to construct one of the following:
\begin{enumerate}
\item a path of type~B while uncovering $2$ edges;
\item an incomplete path $T_1T_2T_3$ of type~C and length $5$ while uncovering
$3$ edges;
\item a red $P_{4}$ while uncovering $2$ edges.
\end{enumerate}
\end{lem}

\begin{proof}
Write $e=uv$ and $f=xy$. Builder chooses the edges $vw$ and $wx$, where $w$ is a new vertex.
If Painter colours both edges red, then $uvwx$ is a red $P_{4}$ and we have achieved~(iii).
Suppose without loss of generality that Painter colours $vw$ blue. If Painter also colours $wx$ 
blue, then $uvwxy$ is a path of type~B and we have achieved~(i). 
If instead Painter colours $wx$ red, then Builder chooses the edge~$tu$.
However Painter colours $tu$, $tuvwxy$ is now a
path of type~C and length~5, taking $T_1 = tuv$, $T_2= vw$ and $T_3 = wxy$.
Moreover, $T_3$ is a red~$P_3$, so $T_1T_2T_3$ is incomplete and we have achieved~(ii).
\end{proof}

We next prove that Builder can use a path of type~B to efficiently extend $Q$.

\begin{lem}
\label{lem:useB-path}Suppose $Q$ is a non-trivial blue path with endpoints~$a$ and~$b$, where~$b$ is incident to a red edge $bc$.
Suppose $vwxyz$ is a path of type~B vertex-disjoint from $V(Q) \cup \{c\}$.
Then, by uncovering at most $3$ edges, Builder can force Painter to construct one of the following: 
\begin{enumerate}
\item a blue path $Q'$ of length $e(Q)+5$ with one endpoint $b'$ incident to a red edge~$b'c'$.
\item a red $P_{4}$.
\end{enumerate}
\end{lem}

\begin{proof}
Builder chooses the edges $bv$, $vy$ and $wz$. If Painter colours any of these
edges red, then $cbvw$, $wvyz$ or $vwzy$ respectively is a red
$P_{4}$ and we have achieved~(ii). Otherwise, $aQbvyxwz$ is a blue
path of length $e(Q)+5$, where $z$ is incident to the red edge $zy$ (see Figure~\ref{fig:useB-path}), and we have achieved (i).
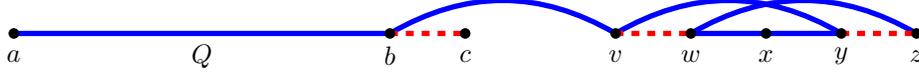
\begin{figure}[tp]
\centering
\begin{tikzpicture}
			\begin{scope}[line width=2pt, blue]
			\draw (10,0) to (12,0);
			\draw (1,0) to (6,0);
			\draw (6,0) to [bend left] (9,0);
			\draw (9,0) to [bend left] (12,0);
			\draw (10,0) to [bend left] (13,0);
			\end{scope}
			
			\begin{scope}[line width=2pt, red, dashed]
			\draw (9,0) to (10,0);
			\draw (12,0) to (13,0);
			\draw (6,0) to (7,0);
			\end{scope}
			
			\begin{scope}[start chain]
			\foreach \i in {9,10,...,13}
			\fill (\i,0) circle (2pt);
 			\end{scope}
 	
			\fill (7,0) circle (2pt);
			\fill (1,0) circle (2pt);
			\fill (6,0) circle (2pt);

			\node at (7,-0.3) {$c$};
			\node at (1,-0.3) {$a$};
			\node at (6,-0.3) {$b$};
			\node at (3.5,-0.3) {$Q$};
	
			\node at (9,-0.3) {$v$};
			\node at (10,-0.3) {$w$};
			\node at (11,-0.3) {$x$};
			\node at (12,-0.3) {$y$};
			\node at (13,-0.3) {$z$};
\end{tikzpicture}
\caption{Extending $Q$ using a path of type~B as in Lemma~\ref{lem:useB-path}.}
\label{fig:useB-path}
\end{figure}
\end{proof}

We now focus on paths of type~C.
We first note the following simple property of such paths, which follows immediately from their definition (Definition~\ref{def:typeC}).

\begin{prop}
\label{prop:lengthbound}
Suppose $T_{1}\dots T_{k}$ is a path of type~C. Then
\[e(T_{1}\dots T_{k}) = 2k-5+e(T_{1})+e(T_{k}).\]
\end{prop}

Let $T_1\ldots T_k$ be an incomplete path of type~C.
We first prove an ancillary lemma, which says that Builder can always extend an incomplete
path of type~C into a slightly longer path of type~C.

\begin{lem}
\label{lem:baby-extendC-path-2}Suppose $T_{1}\dots T_{k}$ is an incomplete path
of type~C and length $\ell$. Then Builder
can force Painter to do one of the following:
\begin{enumerate}
\item for some $i\in\{3,4\}$, extend $T_{1}\dots T_{k}$ to a path $T_{1}'\dots T_{k+2}'$
of type~C and length $\ell+i$ while uncovering $i$ edges.

\item construct a red $P_{4}$ while uncovering at most $4$ edges.
\end{enumerate}
\end{lem}
\begin{proof}
Suppose without loss of generality that $T_{k}= x_k y_k z_k$ is a red $P_3$, where $x_k \in V(T_{k-1})$.
Set $T_{i}'=T_{i}$ for $i\le k$.
Then Builder chooses two edges $uv$ and $vw$, where $u,v$ and $w$ are new vertices.

First suppose Painter colours both edges blue.
Then Builder chooses the edge~$z_k u$.
If Painter colours $z_k u$ red, then $ x_k y_k z_k u $
is a red $P_{4}$ and we have achieved~(ii).
If Painter colours $z_k u$ blue, then set $T'_{k+1} = z_k u v$ and $T'_{k+2} = v w$.
Thus, $T_{1}'\dots T_{k+2}'$ is a path of type~C and length $\ell +3$, and we have achieved (i).

Now suppose that Painter colours both $uv$ and $vw$ red.
Then Builder chooses the edges $z_k t$ and $tu$, where $t$ is a new vertex.
If Painter colours one of these edges red, then $x_k y_k z_k t$ or $t u v w$ is a red $P_4$, respectively, and we have achieved (ii).
If Painter colours both $z_k t$ and $tu$ blue, then set $T'_{k+1} = z_k t u$ and $T'_{k+2} = u v w$.
Thus, $T_{1}'\dots T_{k+2}'$ is a path of type~C and length $\ell + 4$, and we have achieved (i).

Finally, suppose without loss of generality that Painter colours $uv$ blue and $vw$ red.
Then Builder chooses the edges $z_k u$ and $wx$, where $x$ is a new vertex.
If Painter colours $z_k u$ red, then $ x_k y_k z_k u $
is a red $P_{4}$ and we have achieved (ii).
If Painter colours $z_k u$ blue, then set $T'_{k+1} = z_k u v$ and $T'_{k+2} = v w x$.
Thus $T_{1}'\dots T_{k+2}'$ is a path of type~C of length $\ell + 4$, however Painter colours $wx$, and we have achieved~(i).
\end{proof}

By applying Lemma~\ref{lem:baby-extendC-path-2} repeatedly, Builder can extend the path $T_1 T_2 T_3$ of type~C given by Lemma~\ref{lem:findBC-paths} into either a complete path of type~C or an arbitrarily long incomplete path of type~C. Recall from Proposition~\ref{prop:lengthbound} that a path $T_1\dots T_k$ of type~C has length at most $2k-1$.

\begin{cor}
\label{cor:extendC-path}
Let $k_0 \ge 5$ be an odd integer.
Suppose $T_1 T_2 T_3$ is an incomplete path of
type~C and length~$5$.
Then Builder can force Painter to do one of the following:
\begin{enumerate}
\item for some $k,\ell \in \N$, extend $T_1 T_2 T_3$ to a complete path $T'_{1}\dots T'_{k}$ of type~C and length~$\ell$ such that $5 \le k \le k_0$, while uncovering $\ell - 5$ edges.
\item for some $\ell \in \N$, extend $T_1 T_2 T_3$ to an incomplete path $T'_{1}\dots T'_{k_0}$ of type~C and length~$\ell$ while uncovering $\ell - 5$ edges.
\item construct a red $P_{4}$ while uncovering at most $2k_0 - 6$ edges.
\end{enumerate}
\end{cor}

We next prove that Builder can extend $R$ using a path of type~C.

\begin{lem}
\label{lem:useD-path}Suppose $T_{1}\dots T_{k}$ is a path of type~C 
with $k\ge5$ and $e(T_{1}\dots T_{k}) = \ell$. Suppose $R$ is a (possibly trivial) blue path which
is vertex-disjoint from $T_{1}\dots T_{k}$. Then Builder can force
Painter to construct one of the following:
\begin{enumerate}
\item a blue path $R'$ of length $e(R) + (5k-7)/2$ while uncovering $3(k-1)/2$ edges.
This case can only occur if $T_{1}\dots T_{k}$ is incomplete.
\item a blue path $R'$ of length $e(R) + \ell'$ while uncovering at most $7\ell'/5 - \ell$ edges for some $1 \le \ell' \le 5(k - 1)/2$.
This case can only occur if $T_{1}\dots T_{k}$ is complete.
\item a red $P_{4}$ while uncovering at most $3(k-1)/2$ edges.
% Moreover, if $T_{1}\dots T_{k}$ is incomplete, this is achieved by uncovering at most $(3k-7)/2$ edges.
\end{enumerate}
\end{lem}

\begin{proof}
Let $a$ and $b$ be the endpoints of~$R$.
(If $R$ is trivial, then let $a=b$.)%
\COMMENT{AL: had `let $a$ and $b$ be the (not necessarily distinct) endpoints of $R$' but this is not what we want.}
For $i\in\{3,5,\ldots,k-2\}$, write $T_{i}=x_{i}y_{i}z_{i}$ where $x_{i}\in V(T_{i-1})$ and $z_{i}\in V(T_{i+1})$.
Thus $x_{i}y_{i}z_{i}$ is a red $P_{3}$ for each $i \in \{3,5,\ldots,k-2\}$. 
Builder chooses the set
\[
F_{1}=\{x_{3}a,bz_{3},x_{5}c_{1},c_{1}z_{5},x_{7}c_{2},c_{2}z_{7},\ldots,x_{k-2}c_{\frac{k-5}{2}},c_{\frac{k-5}{2}}z_{k-2}\}
\]
of edges, where $c_{1},\dots,c_{\frac{k-5}{2}}$ are new vertices.
Note that 
\begin{equation}
|F_{1}|=2 + 2\cdot\frac{k-5}{2} = k-3 < \frac{3(k-1)}{2}.\label{eq:F1bound}
\end{equation}
If Painter colours an edge in $F_{1}$ red, say $x_{i}w$ or $wz_{i}$ for some integer $i$
and some vertex $w$, then $z_{i}y_{i}x_{i}w$ or $wz_{i}y_{i}x_{i}$ respectively is a red $P_{4}$.
So in this case we have achieved~(iii). 

Now suppose Painter colours all edges in $F_{1}$ blue. Then we have
obtained a blue path
\[
S_{1}= T_{2}x_{3} a R b z_{3} T_{4} x_{5} c_{1} z_{5} T_{6} x_{7} c_2 z_7 \dots T_{k-3} x_{k-2} c_{\frac{k-5}{2}} z_{k-2} T_{k-1}.
\]
Note that $S_{1}$ has length
\begin{equation}
\begin{aligned}\label{eqn:S1}
e(S_{1}) & =e(T_{2})+e(T_{4})+\dots+e(T_{k-1})+|F_{1}|+e(R) \\
 & = \left(2\cdot \frac{k-3}{2} + 1\right) + (k-3) + e(R) = e(R) + 2k-5,
\end{aligned}
\end{equation}
where the second equality follows from \eqref{eq:F1bound}.\COMMENT{This equation has changed as we discussed. -JAL 15/10}

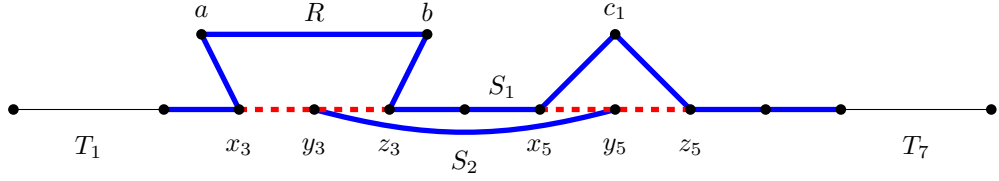
\begin{figure}[t]
\centering
\begin{tikzpicture}
			\begin{scope}[line width=2pt, blue]
			\draw (3,0) to (4,0);
			\draw (6,0) to (8,0);
% 			\draw (1,0) to (2,0);
			\draw (12,0) to (10,0);
% 			\draw (13,0) to (14,0);
		
				% S_1
			\draw (3.5,1) to (6.5,1);
			\draw (3.5,1) to (4,0);
			\draw (6.5,1) to (6,0);
			\draw (8,0) to (9,1) to (10,0);
			
				%S_2
			\draw (5,0) to [out=-15, in =-165] (9,0);
			\end{scope}
			
			\begin{scope}[line width=2pt, red, dashed]
		\draw (4,0) to (6,0);
% 			\draw (3,0) to (2,0);
			\draw (8,0) to (10,0);
% 			\draw (13,0) to (12,0);
			\end{scope}
		
			\begin{scope}
			\draw (1,0) to (3,0);
			\draw (12,0) to (14,0);
			\end{scope}
	
			\begin{scope}[start chain]
			\foreach \i in {3,...,12}
			\fill (\i,0) circle (2pt);
			\end{scope}
			\fill (1,0) circle (2pt);
			\fill (14,0) circle (2pt);
			
 			\node at (2,-0.5) {$T_1$};
% 			\node at (3,-0.5) {$z_1$};

			\node at (4,-0.5) {$x_3$};
			\node at (5,-0.5) {$y_3$};
			\node at (6,-0.5) {$z_3$};

			\node at (8,-0.5) {$x_5$};
			\node at (9,-0.5) {$y_5$};
			\node at (10,-0.5) {$z_5$};

% 			\node at (12,-0.5) {$x_7$};
			\node at (13,-0.5) {$T_7$};
			
			\fill (3.5,1) circle (2pt);
			\fill (6.5,1) circle (2pt);
			\fill (9,1) circle (2pt);
			
			\node at (3.5,1.3) {$a$};
			\node at (6.5,1.3) {$b$};
			\node at (9,1.3) {$c_1$};
			\node at (5,1.3) {$R$};
			
			\node at (7.5,0.3) {$S_1$};
			\node at (7,-0.7) {$S_2$};
\end{tikzpicture}
\caption{Structure of $S_1$ and $S_2$ in Lemma~\ref{lem:useD-path} for a path $T_1\dots T_7$ of type C.}
\label{fig:S1S2}
\end{figure}

Builder now chooses the set
\[
F_{2}=\{ y_3y_5 , y_5y_7, \ldots, y_{k-4}y_{k-2}\}
\]
of edges.
Note that $|F_2| = (k-5)/2$, so by \eqref{eq:F1bound} we have uncovered
\begin{align}
|F_{1}|+|F_{2}|=k-3+\frac{k-5}{2}=\frac{3k-11}{2} \label{eqn:F1+F2}
\end{align}
edges in total so far.
If Painter colours an edge in $F_2$ red, say $y_{i}y_{i+2}$ for some $i \in \{3,5,\ldots, k-4\}$, then
$z_{i}y_{i}y_{i+2}x_{i+2}$ is a red $P_{4}$.
So in this case we have achieved~(iii). Suppose Painter
colours all edges in $F_{2}$ blue. Then we have obtained a blue path
\[
S_{2}=y_{k-2}y_{k-4}\dots y_{5}y_{3}.
\]
Note that $S_{2}$ has length $|F_2| = (k-5)/2$.
Moreover, $S_1$ and $S_2$ are vertex-disjoint (see Figure~\ref{fig:S1S2}) and by~\eqref{eqn:S1} we have
\begin{align} 
e(S_{1})+e(S_{2}) & = e(R) + 2k-5 +\frac{k-5}2 = e(R) + \frac{5 ( k - 3)}{2}. \label{eqn:S1+S2}
\end{align}

Our aim is now to join $S_1$ and $S_2$ together to form $R'$. The way in which we do this depends on the structure of $T_1$ and $T_k$.

\begin{figure}[t]
\centering
\begin{tikzpicture}
			\begin{scope}[line width=2pt, blue]
			\draw (3,0) to (4,0);
			\draw (6,0) to (8,0);
			\draw (12,0) to (10,0);
% 			\draw (13,0) to (12,0);
				% S_1
			\draw (3.5,1) to (6.5,1);
			\draw (3.5,1) to (4,0);
			\draw (6.5,1) to (6,0);
			\draw (8,0) to (9,1) to (10,0);
			
				%S_2
			\draw (5,0) to [out=-15, in =-165] (9,0);
			\draw (2,0) to [out=15, in =165] (9,0);
			\draw (1,0) to [out=-15, in =-165] (5,0);
			\draw (1,0) to (2,1) to (3,0);
			\end{scope}
			
			\begin{scope}[line width=2pt, red, dashed]
			\draw (4,0) to (6,0);
			\draw (8,0) to (10,0);
			\draw (3,0) to (1,0);
			
			\end{scope}
			
			\begin{scope}[start chain]
			\foreach \i in {1,...,12}
			\fill (\i,0) circle (2pt);
			\end{scope}
			
			\fill (14,0) circle (2pt);
			\draw (12,0) to (14,0);
			\fill (2,1) circle (2pt);
			\node at (2,1.3) {$u$};
			
			\node at (1,-0.5) {$x_1$};
			\node at (2,-0.5) {$y_1$};
			\node at (3,-0.5) {$z_1$};

			\node at (4,-0.5) {$x_3$};
			\node at (5,-0.5) {$y_3$};
			\node at (6,-0.5) {$z_3$};

			\node at (8,-0.5) {$x_5$};
			\node at (9,-0.5) {$y_5$};
			\node at (10,-0.5) {$z_5$};

% 			\node at (12,-0.5) {$x_7$};
			\node at (13,-0.5) {$T_7$};
% 			\node at (14,-0.5) {$z_7$};
			
			\fill (3.5,1) circle (2pt);
			\fill (6.5,1) circle (2pt);
			\fill (9,1) circle (2pt);
			
			\node at (3.5,1.3) {$a$};
			\node at (6.5,1.3) {$b$};
			\node at (9,1.3) {$c_1$};
			\node at (5,1.3) {$R$};
			
			\node at (1,1.3) {(i)};
\end{tikzpicture}
\vspace{2mm}
\begin{tikzpicture}
			\begin{scope}[line width=2pt, blue]
			\draw (3,0) to (4,0);
			\draw (6,0) to (8,0);
			\draw (3,0) to (2,0);
			\draw (12,0) to (10,0);
			\draw (13,0) to (12,0);
				% S_1
			\draw (3.5,1) to (6.5,1);
			\draw (3.5,1) to (4,0);
			\draw (6.5,1) to (6,0);
			\draw (8,0) to (9,1) to (10,0);
			
				%S_2
			\draw (5,0) to [out=-15, in =-165] (9,0);
			\draw (2,0) to [out=-15, in =-165] (5,0);
			\end{scope}
			
			\begin{scope}[line width=2pt, red, dashed]
			\draw (4,0) to (6,0);
			\draw (8,0) to (10,0);
			\end{scope}
			
			\begin{scope}[start chain]
			\foreach \i in {2,...,13}
			\fill (\i,0) circle (2pt);
			\end{scope}
			
			\node at (2,-0.5) {$x_1$};
% 			\node at (2,-0.5) {$y_1$};
			\node at (3,-0.5) {$z_1$};

			\node at (4,-0.5) {$x_3$};
			\node at (5,-0.5) {$y_3$};
			\node at (6,-0.5) {$z_3$};

			\node at (8,-0.5) {$x_5$};
			\node at (9,-0.5) {$y_5$};
			\node at (10,-0.5) {$z_5$};

			\node at (12,-0.5) {$x_7$};
			\node at (13,-0.5) {$z_7$};
% 			\node at (14,-0.5) {$z_7$};
			
			\fill (3.5,1) circle (2pt);
			\fill (6.5,1) circle (2pt);
			\fill (9,1) circle (2pt);
			
			\node at (3.5,1.3) {$a$};
			\node at (6.5,1.3) {$b$};
			\node at (9,1.3) {$c_1$};
			\node at (5,1.3) {$R$};
			
			\node at (1,1.3) {(ii)};
\end{tikzpicture}
\vspace{2mm}
\begin{tikzpicture}
			\begin{scope}[line width=2pt, blue]
			\draw (3,0) to (4,0);
			\draw (6,0) to (8,0);
			\draw (3,0) to (2,0);
			\draw (12,0) to (10,0);
			\draw (13,0) to (14,0);
				% S_1
			\draw (3.5,1) to (6.5,1);
			\draw (3.5,1) to (4,0);
			\draw (6.5,1) to (6,0);
			\draw (8,0) to (9,1) to (10,0);
			
				%S_2
			\draw (5,0) to [out=-15, in =-165] (9,0);
			
			\draw (9,0) to [out=-15, in =-165] (12,0);
			\draw (5,0) to [out=15, in =165] (13,0);
			\end{scope}
			
			\begin{scope}[line width=2pt, red, dashed]
			\draw (4,0) to (6,0);
			\draw (8,0) to (10,0);
			\draw (13,0) to (12,0);
			\end{scope}
			
			\begin{scope}[start chain]
			\foreach \i in {2,...,14}
			\fill (\i,0) circle (2pt);
			\end{scope}
			
			\node at (2,-0.5) {$x_1$};
% 			\node at (2,-0.5) {$y_1$};
			\node at (3,-0.5) {$z_1$};

			\node at (4,-0.5) {$x_3$};
			\node at (5,-0.5) {$y_3$};
			\node at (6,-0.5) {$z_3$};

			\node at (8,-0.5) {$x_5$};
			\node at (9,-0.5) {$y_5$};
			\node at (10,-0.5) {$z_5$};

			\node at (12,-0.5) {$x_7$};
			\node at (13,-0.5) {$y_7$};
			\node at (14,-0.5) {$z_7$};
			
			\fill (3.5,1) circle (2pt);
			\fill (6.5,1) circle (2pt);
			\fill (9,1) circle (2pt);
			
			\node at (3.5,1.3) {$a$};
			\node at (6.5,1.3) {$b$};
			\node at (9,1.3) {$c_1$};
			\node at (5,1.3) {$R$};
			
			\node at (1,1.3) {(iii)};
\end{tikzpicture}
\vspace{2mm}
\begin{tikzpicture}
			\begin{scope}[line width=2pt, blue]
			\draw (3,0) to (4,0);
			\draw (6,0) to (8,0);
			\draw (1,0) to (2,0);
			\draw (12,0) to (10,0);
			\draw (13,0) to (14,0);
				% S_1
			\draw (3.5,1) to (6.5,1);
			\draw (3.5,1) to (4,0);
			\draw (6.5,1) to (6,0);
			\draw (8,0) to (9,1) to (10,0);
			
				%S_2
			\draw (5,0) to [out=-15, in =-165] (9,0);
			
			\draw (3,0) to [out = 15, in = 165] (13,0);
			\draw (9,0) to [out=-15, in =-165] (12,0);
			\draw (2,0) to [out=-15, in =-165] (5,0);
			\end{scope}
			
			\begin{scope}[line width=2pt, red, dashed]
			\draw (4,0) to (6,0);
			\draw (3,0) to (2,0);
			\draw (8,0) to (10,0);
			\draw (13,0) to (12,0);
			\end{scope}
			
			\begin{scope}[start chain]
			\foreach \i in {1,2,...,14}
			\fill (\i,0) circle (2pt);
			\end{scope}
			
			\node at (1,-0.5) {$x_1$};
			\node at (2,-0.5) {$y_1$};
			\node at (3,-0.5) {$z_1$};

			\node at (4,-0.5) {$x_3$};
			\node at (5,-0.5) {$y_3$};
			\node at (6,-0.5) {$z_3$};

			\node at (8,-0.5) {$x_5$};
			\node at (9,-0.5) {$y_5$};
			\node at (10,-0.5) {$z_5$};

			\node at (12,-0.5) {$x_7$};
			\node at (13,-0.5) {$y_7$};
			\node at (14,-0.5) {$z_7$};
			
			\fill (3.5,1) circle (2pt);
			\fill (6.5,1) circle (2pt);
			\fill (9,1) circle (2pt);
			
			\node at (3.5,1.3) {$a$};
			\node at (6.5,1.3) {$b$};
			\node at (9,1.3) {$c_1$};
			\node at (5,1.3) {$R$};
			
			\node at (1,1.3) {(iv)};
\end{tikzpicture}
\caption{Extending a blue path $R$ with a path $T_1\dots T_7$ as in cases 1 through 4 (respectively) of Lemma~\ref{lem:useD-path}.}
\label{fig:useD-path}
\end{figure}
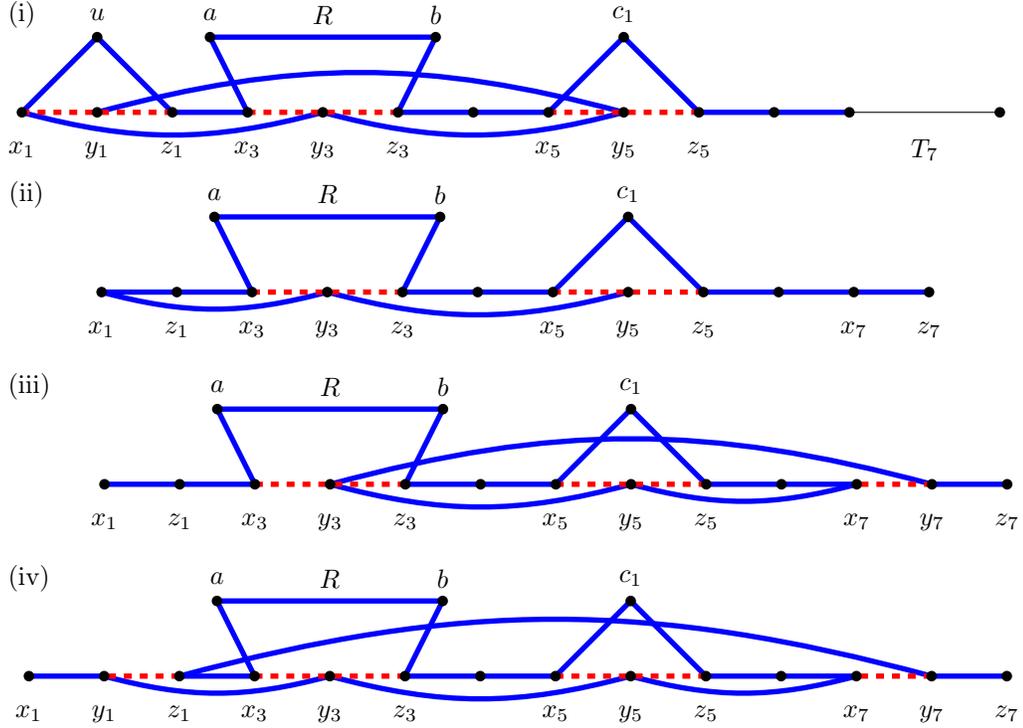

\medskip\noindent \textbf{Case 1: } $T_1\ldots T_k$ is incomplete.

Without loss of generality we may assume that $T_1$ is a red $P_3$, say $x_1 y_1 z_1$ with $z_1 \in V(T_2)$.
Builder chooses the edges $y_{1}y_{k-2}$, $y_{3}x_1$, $x_{1}u$ and $uz_{1}$, where~$u$ is a new vertex.
In total, Builder has uncovered $|F_{1}|+|F_{2}|+4 = 3(k-1)/2$ edges by~\eqref{eqn:F1+F2}.
If Painter colours any of the edges red, then $x_1y_{1}y_{k-2}z_{k-2}$, $y_{3} x_1y_1 z_1$, $z_1 y_1 x_1 u$ or $u z_1 y_1 x_1$ is a red $P_4$, respectively,  and we have achieved~(iii).
Suppose Painter colours them all blue.
Then $R':=y_1 y_{k-2}S_2 y_3 x_1 u z_1 S_1$ is a blue path of length $e(S_1)+e(S_2) +4 = e(R) + (5k-7)/2$ by~\eqref{eqn:S1+S2} (see Figure~\ref{fig:useD-path}(i)) and hence we have achieved~(i).

\medskip\noindent \textbf{Case 2: } $T_1\dots T_k$ is complete and each of $T_1$ and $T_k$ is a blue edge.%
\COMMENT{AL:Cannot use Lemma~\ref{lem:joinpaths}. }

Write $T_1 = x_1z_1$ and $T_k = x_kz_k$ with $z_1 \in V(T_2)$ and $x_k \in V(T_{k-1})$.
First suppose that $k \ge 7$.
Builder chooses the edges $y_{3}x_1$ and $y_{k-2}x_1$.
In total, Builder has uncovered $|F_1|+|F_2|+2=(3k-7)/2$ edges by~\eqref{eqn:F1+F2}. 
If Painter colours both edges red, then $x_3y_3x_1y_{k-2}$ is a red $P_4$ and we have achieved~(iii).
Suppose Painter colours $x_1y_3$ blue.
Then $R':=S_ 2 y_3 x_1 z_1 S_1 x_kz_k$ is a blue path of length $e(S_1)+e(S_2) +3 = e(R) + (5k-9)/2$ by~\eqref{eqn:S1+S2} (see Figure~\ref{fig:useD-path}(ii)).
Writing $\ell' := e(R') - e(R) = (5k-9)/{2}$, Builder has uncovered 
\begin{align*}
\frac{3k-7}{2} < \frac{7}{5}\cdot \frac{5k-9}{2} - (2k-3) = \frac{7\ell'}{5} - \ell
\end{align*}
edges in total, where the last equality follows from Proposition~\ref{prop:lengthbound}.
Hence we have achieved~(ii).
If instead Painter colours $x_1y_{k-2}$ blue, the same argument shows we have achieved~(ii) on replacing $S_2y_3$ by $S_2y_{k-2}$.
So if $k \ge 7$, we are done.

If instead $k = 5$\COMMENT{This can never happen but it's easier to handle the case than it is to exclude it... -JAL 02/09}, 
Builder chooses the edges $y_{3}x_1$ and $ux_1$, where $u$ is a new vertex.
If Painter colours both edges red, then $ux_1y_3z_3$ is a red $P_4$ and we have achieved~(iii).
Suppose instead Painter colours $wx_1$ blue for some $w \in \{u,y_3\}$. Then $R' := wx_1z_1S_1x_5z_5$ is a blue path of length $e(S_1) + e(S_2) + 3$ (as $e(S_2) = 0$) and Builder has uncovered $|F_1|+|F_2|+2$ edges. Thus we have achieved (ii) as above.

\medskip\noindent \textbf{Case 3: } $T_1\dots T_k$ is complete and exactly one of $T_1$ and $T_k$ is a blue edge.

Without loss of generality we may assume that $T_1$ is a blue edge. Let $T_1 = x_1 z_1$ with $z_1 \in V(T_2)$,
and let $T_k = x_ky_kz_k$ with $x_k \in V(T_{k-1})$. Note that $x_ky_k$ is red and $y_k z_k$ is blue.
Builder chooses the edges $x_ky_{k-2}$ and $y_3y_{k}$.
In total, Builder has uncovered $|F_1|+|F_2|+2=(3k-7)/2$ edges by~\eqref{eqn:F1+F2}. 
If Painter colours either $x_ky_{k-2}$ or $y_3y_k$ red, then $y_kx_ky_{k-2}x_{k-2}$ or $x_3 y_3 y_k x_k$ is a red $P_4$ respectively,  and we have achieved~(iii).
Suppose Painter instead colours both edges blue.
Then $R':=x_1 z_1 S_1 x_k y_{k-2} S_2 y_3 y_k z_k$ is a blue path of length $e(S_1)+e(S_2) +4 = e(R) + (5k-7)/2$ by~\eqref{eqn:S1+S2} (see Figure~\ref{fig:useD-path}(iii)).
Writing $\ell' := e(R') - e(R) = (5k-7)/{2}$,
Builder has uncovered 
\begin{align*}
\frac{3k-7}{2} < \frac{7}{5}\cdot \frac{5k-7}{2} - (2k-2) = \frac{7\ell'}{5} - \ell
\end{align*}
edges in total, where the last equality follows from Proposition~\ref{prop:lengthbound}.
Hence we have achieved~(ii).

\medskip\noindent \textbf{Case 4: } $T_1\dots T_k$ is complete and neither $T_1$ nor $T_k$ is a blue edge.

Let $T_{1}=x_{1}y_{1}z_{1}$ and $T_{k}=x_{k}y_{k}z_{k}$ where $z_{1}\in V(T_{2})$ and $x_{k}\in V(T_{k-1})$. Thus $x_{1}y_{1}$ and $y_{k}z_{k}$ are blue, and $y_{1}z_{1}$ and $x_{k}y_{k}$ are
red. Then Builder chooses the edges $y_{k}z_{1}$, $x_{k}y_{k-2}$,
and $y_{3}y_{1}$. 
In total, Builder has uncovered $|F_1|+|F_2|+3=(3k-5)/2$ edges by~\eqref{eqn:F1+F2}.
If Painter colours one of these edges red, then
$x_{k}y_{k}z_{1}y_{1}$, $y_{k}x_{k}y_{k-2}x_{k-2}$ or $z_{3}y_{3}y_{1}z_{1}$
respectively is a red $P_{4}$ and we have achieved~(iii). Suppose
Painter colours them all blue.
Then $R':=z_{k}y_{k}z_{1}S_{1}x_{k}y_{k-2}S_{2}y_{3}y_{1}x_{1}$ is a blue path (see Figure~\ref{fig:useD-path}(iv)) of length $e(S_1)+e(S_2)+5 = e(R)+5(k-1)/2$ by \eqref{eqn:S1+S2}.
Writing $\ell' := e(R') - e(R) = ({5k-5})/{2}$, 
Builder has uncovered 
\begin{align*}
\frac{3k-5}{2} = \frac{7}{5}\cdot \frac{5k-5}{2} - (2k-1) = \frac{7\ell'}{5} - \ell
\end{align*}
edges in total, where the last equality follows from Proposition~\ref{prop:lengthbound}.
We have achieved case~(ii).
\end{proof}

Finally, we consolidate Lemmas~\ref{lem:findBC-paths},~\ref{lem:useB-path} and~\ref{lem:useD-path} and Corollary~\ref{cor:extendC-path} into a single lemma which says that given two independent red edges, Builder can extend either $Q$ or~$R$. As with Lemma~\ref{lem:useblueedges}, in applying Lemma~\ref{lem:userededges} we will take $m$ to be $\ell - e(Q) - e(R)-1$.

\begin{lem} \label{lem:userededges}
Let $m \ge 9$ be an integer.
Let $Q$ and $R$ be blue paths and let $e$ and $f$ be two red edges. 
Suppose that $Q$ is non-trivial and has an endpoint $b$
incident to a red edge $bc$. 
Further suppose that $V(Q)\cup\{c\}$, $R$, $e$ and $f$ are pairwise vertex-disjoint.
Then Builder can force Painter to construct one of the following:
\begin{enumerate}
\item a blue path $Q'$ with one endpoint $b'$ incident to a red edge $b'c'$ such that $e(Q') = e(Q) +5$, while uncovering $5$ edges. Moreover, $R$ is vertex-disjoint from $V(Q')\cup\{c'\}$.

\item  a blue path $R'$ such that $e(R') = e(R) +\ell'$ for some $1 \le \ell' \le m+5$ while uncovering at most $7\ell'/5-2$ edges. Moreover, $R'$ is vertex-disjoint from $V(Q)\cup\{c\}$.

\item  a blue path $R'$ such that $e(R') \ge e(R) + m$ while uncovering at most $7m/5+6$ edges. Moreover, $R'$ is vertex-disjoint from $V(Q)\cup \{c\}$.

\item a red $P_{4}$ while uncovering at most $7m/5+6$ edges.
\end{enumerate}
\end{lem}

\begin{proof}
We first apply Lemma~\ref{lem:findBC-paths} to $e$ and $f$.
If as in Lemma~\ref{lem:findBC-paths}(iii) we obtain a red $P_{4}$
while uncovering 2 edges, then we have achieved~(iv). Suppose we do
not. Then we split into cases depending on Painter's choice.

\medskip\noindent \textbf{Case 1:} We obtain a path $vwxyz$ of type~B while uncovering 2 edges, as in Lemma~\ref{lem:findBC-paths}(i). Moreover, $vwxyz$ is vertex-disjoint from $V(Q)\cup \{c\}$ and $R$.

We apply Lemma~\ref{lem:useB-path} to $Q$ and $vwxyz$. 
Hence we have uncovered at most 5 edges in total.
If we obtain
a red $P_{4}$, then we have achieved~(iv).
Suppose instead we obtain a blue path $Q'$ of length $q+5$ with one endpoint $b'$ incident to a red
edge $b'c'$, where $V(Q')\cup\{c'\}$ is vertex-disjoint from $R$.
Then we have achieved (i). 

\medskip\noindent \textbf{Case 2:} We obtain an incomplete path $T_1T_2T_3$ of type~C and length~5 while uncovering 3 edges, as in Lemma~\ref{lem:findBC-paths}(ii). Moreover, $T_1T_2T_3$ is vertex-disjoint from $V(Q)\cup\{c\}$ and $R$.

Let $k_0$ be the least odd number such that $k_0 \ge (2m+7)/5$.
Since $5k_0 < (2m+7) + 5\cdot 2$, and both $5k_0$ and $2m+17$ are odd integers, we have $k_0 \le 2m/5+3$.
Moreover, $k_0 \ge (2m+7)/5 \ge 5$ since $m \ge 9$.
We apply Corollary~\ref{cor:extendC-path} to $T_1T_2T_3$ and~$k_0$. If
we obtain a red $P_{4}$ while uncovering at most $2k_0 -6$ additional
edges, then we have achieved~(iv).
Suppose we do not. Then we split into subcases depending on Painter's choice.

\medskip\noindent \textbf{Case 2a:} For some $k,\ell \in \N$, we obtain a complete path $T'_{1}\dots T'_{k}$ of type~C and length $\ell$ such that $5 \le k \le k_0$
while uncovering $\ell-5$ additional edges, as in Corollary~\ref{cor:extendC-path}(i).
Moreover, $T'_{1}\dots T'_{k}$ is vertex-disjoint from $V(Q)\cup\{c\}$ and $R$.

We now apply Lemma~\ref{lem:useD-path} to $T'_{1}\dots T'_{k}$ and $R$.
Suppose we obtain a blue path~$R'$ with length $e(R)+\ell'$, where 
\[\ell' \le \frac{5(k-1)}{2} \le \frac{5(k_0-1)}{2} \le \frac{5}{2}\cdot \left(\frac{2m}{5}+2\right) = m+5,\] 
while uncovering at most $7\ell'/5 - \ell$ edges as in Lemma~\ref{lem:useD-path}(ii).
Note that $R'$ is vertex-disjoint from $V(Q)\cup\{c\}$.
In total we have uncovered at most $3 + (\ell-5) + (7\ell'/5-\ell) = 7 \ell'/5-2$ edges, so we have achieved~(i).

Suppose instead we obtain a red~$P_{4}$ while uncovering at most $3(k-1)/2$ edges as in Lemma~\ref{lem:useD-path}(iii).
Note that $\ell \le 2 k_0 -1$ by Proposition~\ref{prop:lengthbound}.
In total we have therefore uncovered at most
\begin{align}
3 + (\ell-5) + \frac{3(k_0-1)}{2} \le \frac{7k_0 -9}{2} \le \frac{7}{2}\cdot\left(\frac{2m}{5}+3\right) - \frac{9}{2} = \frac{7m}{5}+6\label{eqn:l1}
\end{align}
edges, and thus we have achieved~(iv).

\medskip\noindent \textbf{Case 2b:} For some $\ell \in \N$, we obtain an incomplete path $T'_{1}\dots T'_{k_0}$ of type~C and length $\ell$
while uncovering $\ell-5$ additional edges, as in Corollary~\ref{cor:extendC-path}(ii).
Moreover, $T'_{1}\dots T'_{k_0}$ is vertex-disjoint from $V(Q)\cup\{c\}$ and $R$.

We apply Lemma~\ref{lem:useD-path} to $T'_{1}\dots T'_{k_0}$ and $R$.
Whatever the outcome, we uncover at most $3(k_0-1)/2$ edges.
We have therefore uncovered at most $7m/5+6$ edges in total, as in~\eqref{eqn:l1}.
If we obtain a red~$P_{4}$ as in Lemma~\ref{lem:useD-path}(iii), then we have achieved~(iv).
Hence we may assume that we obtain a blue path $R'$ of length 
\[e(R) +\frac{5k_0 -7}{2} \ge e(R) + \frac{5}{2} \cdot \frac{2m+7}{5} - \frac{7}{2} = e(R)+ m,\] 
as in Lemma~\ref{lem:useD-path}(i). (The inequality follows from the definition of $k_0$.)
We have therefore achieved~(iii).
\end{proof}

\subsection{\label{sub:fullalgo}An upper bound on $\tilde{r}(P_{4},P_{\ell+1})$ for $\ell \ge 3$}

We now use Lemmas~\ref{lem:joinpaths}, \ref{lem:useblueedges} and~\ref{lem:userededges}
to bound $\tilde{r}(P_4,P_{\ell+1})$ above in Theorem~\ref{thm:P4-upper}.
Together with Theorem~\ref{thm:lowerbound}, this will imply Theorem~\ref{thm:pathresultsk=3}.

Recall that Builder's strategy is to extend blue paths $Q$ and $R$ using independent edges.
For the remainder of the section, we denote the graph Builder has uncovered by $G$.
In order to keep track of the lengths of $Q$ and $R$ and the number of independent edges available, we introduce the following notation.

\begin{defn} \label{def:track}%
Given $q,r,n_{\textnormal{blue}},n_{\textnormal{red}} \in \mathbb{N}_0$, we say that a graph $G$ contains a \emph{$(q,r,n_{\textnormal{blue}},n_{\textnormal{red}})$-structure} if  it satisfies the following properties:
\begin{enumerate}
\item [(P1)] $G$ contains a (possibly trivial) blue path $Q$ of length $q$ with one endpoint $b$ incident to a red edge $bc$.
\item [(P2)] $G$ contains a (possibly trivial) blue path $R$ of length $r$.
\item [(P3)] $G$ contains a set $F$ of independent edges containing $n_{\textnormal{blue}}$ blue edges and $n_{\textnormal{red}}$ red edges.
\item [(P4)] $V(Q)\cup\{c\}$, $R$ and $F$ are pairwise vertex-disjoint. 
\end{enumerate}
\end{defn}
This notation substantially simplifies the statements of Lemmas~\ref{lem:joinpaths}, \ref{lem:useblueedges} and~\ref{lem:userededges}.
The corresponding statements are as follows.

\begin{cor}
\label{cor:joinpaths}
Let $q,r,n_{\textnormal{red}},n_{\textnormal{blue}} \in \mathbb{N}_0$.
Suppose $G$ is a graph containing a $(q,r,n_{\textnormal{blue}},n_{\textnormal{red}})$-structure.
Then Builder can force Painter to construct a graph $G'\supseteq G$ with $e(G') \le e(G) + 2$ such that 
$G'$ contains a $(q+r+1,0,n_{\textnormal{blue}},n_{\textnormal{red}})$-structure or a red $P_4$.
\end{cor}

\begin{cor}
\label{cor:useblueedges}
Let $m,q,r,n_{\textnormal{red}} \in \mathbb{N}_0$ with $q,m \ge 1$.
Suppose $G$ is a graph containing a $(q,r,2,n_{\textnormal{red}})$-structure.
Then Builder can force Painter to construct a graph $G' \supseteq G$ such that one of the following holds:
\begin{enumerate}
\item $G'$ contains a $(q+\ell',r,n_{\textnormal{blue}},n_{\textnormal{red}})$-structure and $e(G') = e(G) + \ell'$ for some $3 \le \ell' \le m+3$ and some $n_{\textnormal{blue}} \in \mathbb{N}_0$.
Moreover, if $3 \le \ell' < 5 \le m$, then we may take $n_{\textnormal{blue}} = 1$.
\item $G'$ contains a red $P_4$ and $e(G') \le e(G) + m+3$.
\end{enumerate}
\end{cor}

\begin{cor}
\label{cor:userededges}
Let $m,q,r,n_{\textnormal{blue}} \in \mathbb{N}_0$ with $q \ge 1$ and $m \ge 9$.
Suppose $G$ is a graph containing a $(q,r,n_{\textnormal{blue}},2)$-structure.
Then Builder can force Painter to construct a graph $G' \supseteq G$ such that one of the following holds:
\begin{enumerate}
%\item There exist $q'$ and $r'$ such that writing $q'+r' = q+r+\ell'$, the following hold. $G'$ contains a $(q',r',n_{\textnormal{blue}},0)$-structure, we have $1 \le \ell' \le m+5$, and $e(G') \le e(G) + 7\ell'/5 - 2$.
\item $e(G') = e(G) + 5$ and $G'$ contains a $(q+5,r,n_{\textnormal{blue}},0)$-structure.
\item There exists $1 \le \ell' \le m+5$ such that $e(G') \le e(G) + 7 \ell'/5-2$ and $G'$ contains a $(q,r+\ell',n_{\textnormal{blue}},0)$-structure.
\item $e(G') \le e(G) + 7m/5 + 6$ and $G'$ contains a $(q,r+m,n_{\textnormal{blue}},0)$-structure.
\item $e(G') \le e(G) + 7m/5 + 6$ and $G'$ contains a red $P_4$.
\end{enumerate}
\end{cor}

\begin{thm}
\label{thm:P4-upper}For all $\ell\in\N$, we have $\tilde{r}(P_{4},P_{\ell+1}) \le (7\ell +52)/5 $.\end{thm} 

\begin{proof}
Our aim is to show that Builder can construct a graph $G$ with $e(G) \le (7\ell+52) /5 $ containing a red $P_4$ or a blue $P_{\ell+1}$.%

We first obtain an initial blue path $Q$ with one endpoint incident to a red edge. We claim that either 
Builder can construct a path $xySz$ of type~A with $e(S) < \ell$, while uncovering at most $(7e(S)+4)/5$ edges, or we are done. 
We proceed as follows. Builder chooses an edge $e = uv$.
First suppose Painter colours $uv$ blue.
Then apply Lemma~\ref{lem:findA-path} to $uv$, taking $m=\ell$.
If we find a blue $P_{\ell+1}$ while uncovering $\ell-1$ additional
edges, then since we have uncovered $\ell$ edges in
total we are done.
Suppose instead we find a path $xySz$ of type~A with $e(S) < \ell$, while uncovering $e(S)$ additional
edges in the process.
Then in total Builder has uncovered $e(S)+1 < (7 e(S) +4)/5$ edges, as desired.

Suppose instead Painter colours $uv$ red.
Then Builder chooses the edge~$vx$, where $x$ is a new vertex.
If Painter colours $vx$ blue, then $uvx$ is a path of type~A constructed while uncovering $2<(7+4)/5$ edges in total.
If Painter colours $vx$ red, then Builder chooses the edges~$t u$, $uw$ and $wx$, where $t$ and $w$ are new vertices.
If Painter colours any of these edges red, then $t u v x$, $x v u w$ or $w x v u$ respectively is a red $P_{4}$ and we are done.
Otherwise, $tuwxv$ is a path of type~A (taking $S= tuwx$), constructed while uncovering $5=(7\cdot 3+4)/5$ edges in total.
Therefore, we may assume that Builder has constructed a path $xySz$ of type~A with $e(S) < \ell$ while uncovering at most $(7 e(S)+4)/5$ edges as claimed.%
\COMMENT{This startup is obviously inefficient, but it only costs us $4/5$ and making it more efficient would be extremely painful.}

Let $G_{0}$ be the graph consisting of all edges uncovered so far.
Thus $G_0$ contains a $(q_0,0,0,0)$-structure for some $1\le q_0 < \ell$, and $e(G_0) \le (7 q_0+4)/5$.
Suppose that for some $i\ge0$, Builder has already constructed a graph $G_i$ such that there exist $q_i,r_i,n_{\textnormal{blue},i},n_{\textnormal{red},i} \in \N_0$ satisfying the following properties:
\begin{enumerate}
\item [(G1)] $G_{i}\subseteq K_{\N}$ is the graph of all uncovered edges.
\item [(G2)] $G_i$ contains a $(q_i,r_i,n_{\textnormal{blue},i},n_{\textnormal{red},i})$-structure, and $q_i > 0$.
\item [(G3)] $q_i+ r_i \le \ell+4$.
\item [(G4)] $n_{\textnormal{red},i} , n_{\textnormal{blue},i} \le 1$. 
\item [(G5)] $e(G_{i})\le  (7(q_i+r_i)+4)/5 + n_{\textnormal{blue},i} + n_{\textnormal{red},i}$.
\end{enumerate}
\noindent Note that (G1)--(G5) hold for $i=0$. 
We are going to show that Builder can force a graph $G_{i+1} \supseteq G_i$ such that one of the following holds:
\begin{enumerate}
\item[(a)] $G_{i+1}$ contains a red $P_4$ or a blue $P_{\ell+1}$ and $e(G_{i+1}) \le (7\ell+52)/5$.
\item[(b)] there exist $q_{i+1},r_{i+1},n_{\textnormal{blue},i+1},n_{\textnormal{red},i+1} \in \N_0$ such that $q_{i+1}+ r_{i+1} > q_i+ r_i$ and $G_{i+1}$, $q_{i+1},r_{i+1},n_{\textnormal{blue},i+1}$ and $n_{\textnormal{red},i+1}$ together satisfy (G1)--(G5).
\end{enumerate}
If (a) holds, we are done. If (b) holds, then Builder can repeat the algorithm to obtain $G_{i+2}$.
We then simply repeat the process until it terminates, which must happen by (G3) (since $q_{i+1}+r_{i+1} > q_i+r_i$ whenever these quantities are defined).
It therefore remains only to prove that forcing such a graph is possible.

Let $m = \ell -q_i-r_i - 1$.
We split into cases depending on the values of $q_i,r_i,n_{\textnormal{blue},i}$ and $n_{\textnormal{red},i}$.

\medskip\noindent \textbf{Case 1: } $q_i + r_i \ge \ell - 1$.

In this case, we may simply join our two blue paths together to achieve (a).
Apply Corollary~\ref{cor:joinpaths} to $G_i$.
Builder obtains a graph $G_{i+1} \supseteq G_i$ with
\begin{equation*}
e(G_{i+1}) = e(G_i) +2 \stackrel{\textnormal{(G5)}}{\le} \frac{7(q_i+r_i) +4}5 +  n_{\textnormal{blue},i} + n_{\textnormal{red},i}+2
	\stackrel{\textnormal{(G3)},\textnormal{(G4)}}{\le} \frac{7\ell+52}{5}.
\end{equation*}
Moreover, $G'$ contains a red~$P_4$ or a blue $P_{\ell+1}$, so we have achieved (a).

\medskip\noindent \textbf{Case 2: } $\ell - 9 \le q_i + r_i \le \ell -2$, so that $1 \le m \le 8$.

In this case, it is more efficient to naively extend our paths to the right combined length and join them than it is to apply our normal extension methods and potentially end up with paths longer than we need.
Builder will force a red $P_4$ or a blue $P_{\ell+1}$ as follows.
Apply Corollary~\ref{cor:joinpaths} to~$G_i$ to obtain a graph $G' \supseteq G_i$ with $e(G') = e(G_i)+2$.
Note that $G'$ contains a red $P_4$ or a $(q_i+ r_i +1,0,n_{\textnormal{blue},i},n_{\textnormal{red},i})$-structure.
By repeating the process at most $m$ additional times, Builder obtains a graph $G'' \supseteq G' \supseteq G_i$, where%
\begin{align*}
e(G'') & \,\,\le\,\, e(G)+ 2m+2 \stackrel{\textnormal{(G5)}}{\le} \frac{7(q_i+r_i)+4}{5} + n_{\textnormal{blue},i} + n_{\textnormal{red},i} + 2m+2  \\
& \stackrel{\textnormal{(G4)}}{\le} \frac{7(\ell - m - 1)+4}{5} + 2 + 2m+2 = \frac{7\ell}{5}+\frac{3m+17}{5} \le \frac{7\ell + 41}{5},
\end{align*}
such that $G''$ contains a red $P_4$ or a $(q_i+ r_i +m+1,0,n_{\textnormal{blue},i},n_{\textnormal{red},i})$-structure (which contains a blue $P_{\ell+1}$).
Thus we have achieved (a).

\medskip\noindent \textbf{Case 3: } $q_i + r_i \le \ell -10$, so that $m \ge 9$.

In this case, we will extend our paths efficiently using Corollaries~\ref{cor:useblueedges} and~\ref{cor:userededges}.
By choosing at most $3- n_{\textnormal{blue},i} - n_{\textnormal{red},i}$ additional independent edges (on new vertices), Builder obtains a graph $G_i' \supseteq G_i$ containing a $(q_i,r_i,n'_{\textnormal{blue}},n'_{\textnormal{red}})$-structure such that $n_{\textnormal{blue}}'+n_{\textnormal{red}}'\le 3$, either $n_{\textnormal{blue}}' = 2$ or $n_{\textnormal{red}}' = 2$, and
\begin{align} 
e(G_i') \stackrel{\textnormal{(G5)}}{\le} \frac{7(q_i+r_i)+4}5 + n'_{\textnormal{blue}} + n'_{\textnormal{red}}. \label{eqn:G'i}
\end{align}
We split into subcases depending on the values of $n'_{\textnormal{blue}}$ and $n'_{\textnormal{red}}$.

\medskip\noindent \textbf{Case 3a:} $n'_{\textnormal{blue}} = 2$ and $n'_{\textnormal{red}} \le 1$.

We apply Corollary~\ref{cor:useblueedges} to $G'_i$, obtaining a graph $G' \supseteq G'_i$. 
First suppose Corollary~\ref{cor:useblueedges}(i) holds, so that
there exists some $3 \le \ell' \le m+3$ such that $G'$ contains a $(q_i+\ell',r_i,n''_{\textnormal{blue}},n'_{\textnormal{red}})$-structure and
$e(G') = e(G'_i) + \ell'$.
Set $G_{i+1} = G'$, $q_{i+1} = q_i+ \ell'$, $r_{i+1} = r_i$ and $n_{\textnormal{red},i+1} = n'_{\textnormal{red}}$. Set $n_{\textnormal{blue},i+1} = 0$ if $\ell' \ge 5$ and $n_{\textnormal{blue},i+1}=1$ otherwise.
Clearly $q_{i+1}+r_{i+1} > q_i + r_i$, and (G1) and (G4) are satisfied. 
Recall from Corollary~\ref{cor:useblueedges}(i) that if $\ell' < 5 \le m$ then we may take $n''_{\textnormal{blue}} = 1$, so (G2) is satisfied.
We have $q_{i+1}+r_{i+1} \le q_i + m+3 + r_i = \ell +2$, so (G3) is satisfied. 
If $3\le \ell' \le 4$, we have
\begin{align}
e(G') & = e(G'_i) + \ell'  \overset{\eqref{eqn:G'i}}{\le} \frac{7(q_i+r_i)+4}5 + 2 + n'_{\textnormal{red}} + \ell'\nonumber \\
& = \frac{7(q_i+r_i+\ell')+4}{5} - \frac{2\ell'}{5} + 2  + n'_{\textnormal{red}} \nonumber
 \le \frac{7(q_{i+1} + r_{i+1})+4}5 + 1 + n'_{\textnormal{red}} \nonumber \\
& = \frac{7(q_{i+1}+r_{i+1})+4}{5} + n_{\textnormal{blue},i+1} + n_{\textnormal{red},i+1}.
\nonumber
%\label{eq:3abound1}
\end{align}
So (G5) is satisfied and we have therefore achieved (b).
A similar argument holds for the case when $\ell' \ge 5$.
%
%
%If instead $\ell' \ge 5$, then by a calculation similar to the above, we have
%\begin{align}
%e(G') & \overset{\eqref{eqn:G'i}}{\le} \frac{7(q_i+r_i)+4}5 + 2 + n'_{\textnormal{red}} + \ell' \le \frac{7(q_{i+1} + r_{i+1})+4}5 + n'_{\textnormal{red}} \nonumber\\
%& \,\,=\,\, \frac{7(q_{i+1}+r_{i+1})+4}{5} + n_{\textnormal{blue},i+1} + n_{\textnormal{red},i+1}\label{eq:3abound2}.
%\end{align}
%Thus, by \eqref{eq:3abound1} and \eqref{eq:3abound2}, (G5) is satisfied. We have therefore achieved (b).

Suppose instead that Corollary~\ref{cor:useblueedges}(ii) holds, so that $G'$ contains a red $P_4$ and $e(G') \le e(G'_i) + m + 3$.
Then we have
\begin{align*}
e(G')& \overset{\eqref{eqn:G'i}}{\le} \frac{7(q_i+r_i)+4}5 + 2 + n_{\textnormal{red}}' + m +3
	\le\, \frac{2(q_i+r_i) +4}5 + \ell + 5  \le \frac{ 7 \ell+9}5,
\end{align*}
where the final inequality follows since $q_i+r_i \le \ell - 10$.
We have therefore achieved~(a).

\medskip\noindent \textbf{Case 3b:} $n'_{\textnormal{red}} = 2$ and $n'_{\textnormal{blue}} \le 1$.

We apply Corollary~\ref{cor:userededges} to $G'_i$, obtaining a graph $G' \supseteq G'_i$.
Suppose Corollary~\ref{cor:userededges}(i) or (ii) holds.
In either case, it follows that there exist $q'$ and $r'$ such that $G'$ contains a $(q',r',n_{\textnormal{blue}}',0)$-structure and
\begin{align*}
1 \le q' +r' - (q_i + r_i) \le m + 5.
\end{align*}
Write $\ell' = q'+r' - (q_i + r_i)$.
Set $G_{i+1} = G'$, $q_{i+1} =q'$, $r_{i+1} =r'$, $n_{\textnormal{blue},i+1}=n'_{\textnormal{blue}}$ and $n_{\textnormal{red},i+1} = 0$.
Clearly (G1)--(G4) are satisfied, and $q_{i+1} + r_{i+1} > q_i + r_i$.
Moreover, we have
\begin{align*}
e(G_{i+1}) & \le e(G'_i) + \frac{7 \ell'}{5}-2 \overset{\eqref{eqn:G'i}}{\le} \frac{7(q_i+r_i + \ell' ) + 4 }5 + n'_{\textnormal{blue}} \\
&  = \frac{7(q_{i+1}+r_{i+1}) + 4 }5 + n_{\textnormal{blue},i+1} + n_{\textnormal{red},i+1},
\end{align*}
so (G5) is satisfied. We have therefore achieved (b).

Now suppose Corollary~\ref{cor:userededges}(iii) holds, so that $G'$ contains a $(q_i,r_i+m,n_{\textnormal{blue}}',0)$-structure and $e(G') \le e(G'_i) +7m/5+6$.
We apply Corollary~\ref{cor:joinpaths} to~$G'$, obtaining a graph $G''$
such that 
\begin{align*}
	e(G'') & \,\, = \,\, e(G')+2 \le e(G'_i) + \frac{7m}{5} + 8\\
	 & \overset{\eqref{eqn:G'i}}{\le} \frac{7(q_i+r_i+m) + 4}{5} + n'_{\textnormal{blue}} + 10 \le \frac{7 \ell + 52}{5}.
\end{align*}
Moreover, $G''$ contains a red $P_4$ or an $(\ell,0,n_{\textnormal{blue}}',0)$-structure (which contains a blue $P_{\ell+1}$). We have therefore achieved (a).

Finally suppose Corollary~\ref{cor:userededges}(iv) holds, so that $G'$ contains a red $P_4$ and $e(G') \le e(G_i') + 7m/5 + 6$.
Then we have
\begin{align*}
	e(G') \le e(G'_i) + \frac{7m}{5}+6 \overset{\eqref{eqn:G'i}}{\le}  \frac{7(q_i+r_i+m) + 4}{5} + n'_{\textnormal{blue}} + 8 \le \frac{7 \ell + 42}5.
\end{align*}
We have therefore achieved (a).
This completes the proof of the theorem.
\end{proof}

% Finally, we present a general upper bound on $\tilde{r}( C_{k},P_{ \ell+1 })$.

% \begin{prop}
% For all $k ,\ell \in \mathbb{N}$ with $k \ge 3$, we have $\tilde{r}( C_{k},P_{ \ell+1 }) \le  (3k+2) \ell / 2-2 $.
% \end{prop}

% \begin{proof}
% Builder first spends at most $3(  (k-1)\ell+ \ell)/2 - 2 = 3 k \ell / 2 - 2$ rounds forcing Painter to
% construct a red $P_{(k-1)\ell+1}$ or a blue $P_{\ell+1}$.
% (This is possible by Corollary~\ref{cor:3-upper-bound}.)
% We may assume that Painter constructs a red $P_{(k-1)\ell+1} = v_1 \dots v_{(k-1)\ell+1}$ (or else we are done).

% Let $i_j = (j-1)(k-1)+1$ for all $j \in [\ell+1]$.
% Then Builder chooses edges $v_{i_j} v_{ i_{j+1} }$ for all $j \in [\ell]$.
% If Painter colours $v_{i_j} v_{ i_{j+1} }$ red for some $j$, then $v_{i_j} v_{i_j+1} \dots v_{ i_{j+1} } v_{i_j}$ is a red~$C_{k}$.
% If Painter colours all the edges blue, then $v_{i_1} v_{i_2} \dots v_{\ell+1}$ is a blue~$P_{\ell+1}$.
% In total, Builder has chosen at most $3 k \ell / 2 - 2 + \ell = (3k+2) \ell / 2 -2$ edges and the proposition follows.
% \end{proof}

\end{document}